\documentclass[11pt]{amsart}
\usepackage{palatino}
\usepackage{amsmath}
\usepackage{amssymb, graphicx, epsfig}

\newtheorem{theorem}{$\quad$Theorem}[section]
\newtheorem*{theorem*}{$\quad$Theorem 1.1}

\newtheorem{lemma}[theorem]{$\quad$Lemma}

\newtheorem{definition}[theorem]{$\quad$Definition}

\newcommand{\R}{{\mathbb R}}
\newcommand{\C}{{\mathbb C}}
\newcommand{\Z}{{\mathbb Z}}

\newcommand{\pa}{\partial}

\newcommand{\krn}{\operatorname{ker}}

\theoremstyle{definition}
\newtheorem{remark}[theorem]{$\quad$Remark}
\newenvironment{example}{{\bf Example. }}{}

\newcounter{bibno}

\begin{document}

\title[Legendrian contact homology\dots]{Legendrian contact homology \\ in the product of a punctured \\ Riemann surface and the real line}
\author{Johan Bj\"orklund}
\thanks{Address: Uppsala University, Box 480, 75106 Uppsala, SWEDEN}
\date{2015}
\thanks{E-Mail:bjorklund@math.uu.se}

\begin{abstract}We give a combinatorial description of the Legendrian differential graded algebra associated to a Legendrian knot in $P\times\R$, where $P$ is a punctured Riemann surface. As an application we show that for any integer $k$ and any homology class $h\in H_1(P\times\R)$ there are $k$ Legendrian knots all representing $h$ which are pairwise smoothly isotopic through a formal Legendrian isotopy but which lie in mutually distinct Legendrian isotopy classes.   
\end{abstract}

\maketitle

\clearpage

\section{Introduction}

In this paper we study Legendrian knots in $P\times\R$, where $P$ is a punctured Riemann surface. Here the symplectic form $\omega$ on the Riemann surface is exact, $\omega=d\theta$ and the contact form on $P\times\R$ is $\alpha=dz-\theta$, where $z$ is a coordinate along the $\R$-factor, and a knot is said to be \emph{Legendrian} if it is everywhere tangent to the contact distribution $\xi=\krn(\alpha)$. The \emph{Reeb vector field} $R$ of a contact form $\alpha$ is characterized by $d\alpha(R,\cdot)=0$ and $\alpha(R)=1$. In the case $P\times\R, \alpha=dz-\theta, R=\pa_z$.

Note that the differential of the {\em Lagrangian projection} $\pi\colon P\times\R\to P$ is an isomorphism when restricted to the contact planes in $\xi$. Pulling back the complex structure on $P$ to $\xi$ we get a complex structure $J$ compatible with $d\alpha$. Let $K$ be an oriented Legendrian knot. Then $K$ comes equipped with an induced framing $E_K=(e_1,e_2,e_3)$, where $e_1$ is the tangent vector of $K$, $e_2=Je_1$, and $e_3=\pa_z$. We say that two Legendrian knots $K_0$ and $K_1$ are \emph{formally Legendrian isotopic} if there exists a smooth isotopy $K_t$ of framed knots with framing $E_t=(e_1^t,e_2^t,e_3^t)$ such that $e_1^t$ and $e^t_2$ lie in $\xi$, $e_3^t$ does not lie in $\xi$, and such that $E_s=E_{K_s}$, $s=0,1$. Furthermore, we say that $K_0$ and $K_1$ are \emph{Legendrian isotopic} if there exists a smooth isotopy $K_t$ such that $K_{t_0}$ is a Legendrian knot for each $t_0\in[0,1]$. Any Legendrian isotopy is a formal Legendrian isotopy.     

Chekanov \cite{CHEK} and Eliashberg \cite{ELIA}, showed that there exist formally Legendrian isotopic knots in $\R^{2}\times\R$ which are not Legendrian isotopic using Legendrian contact homology. Both proofs utilized linearized contact homology, a theory which was later incorporated in the theoretical framework by Eliashberg, Givental and Hofer in \cite{EGA} introducing symplectic field theory. Legendrian contact homology associates a differential graded algebra (DGA) to a Legendrian knot $K$. The DGA is generated by Reeb chords on $K$, i.e. flow lines of $R$ starting and ending on $K$ and the differential is given by a holomorphic curve count in the symplectization of the contact manifold. The quasi-isomorphism type of the DGA (in particular its homology) is invariant under Legendrian isotopy. In \cite{ekholm1} Ekholm, Etnyre and Sullivan worked out the details of Legendrian contact homology in the case of a contact manifold of the form $P\times\R$ where $P$ is an exact symplectic manifold of any even dimension $2n$. If $\Lambda\subset P\times\R$ then $\pi\colon \Lambda\to P$ is a Lagrangian immersion and Reeb chords of $\Lambda$ corresponds to double points of this immersion. In \cite{ekholm1}, a complex structure on the contact planes which is pulled back from an almost complex structure on $P$ were used. For such a complex structure, holomorphic disks in $P\times\R$ with boundary on $\Lambda\times\R$ can be described in terms of holomorphic disks in $P$ with boundary on $\pi(\Lambda)$ and the DGA of $\Lambda$ was shown to invariant under Legendrian isotopies up to stable tame isomorphism. 
In dimension $2+1$ the contact homology can be defined combinatorially by using the Riemann mapping theorem. This was first observed by Etnyre, Ng and Sabloff in \cite{SABLOFF2}. In this paper we describe how to compute the Legendrian contact homology combinatorially when $P$ is a punctured Riemann surface. Similar situations has also been studied by other authors, e.g. Sabloff in \cite{SABLOFF1} who studies Legendrian contact homology in circle bundles and Ng and Traynor in \cite{NGTREY} where they give a combinatorial interpretation of contact homology in $J^1(S^1)$.  
If $K$ is a Legendrian knot in $P\times\R$ then the DGA of $K$ is generated by crossings of the knot diagram of $K$ in $P$ and the differential can be computed by counting rigid holomorphic disks with boundary on the knot diagram. By the Riemann mapping theorem, such disks correspond to immersed polygons in $P$ with boundary on the knot diagram. We give detailed definitions of this DGA in Section \ref{DGAsection}. In order to construct and work with Legendrian knots in $\R^{2}\times\R$ it is often more convenient to work with knot diagrams in the \emph{front projection}: if $\theta=y\,dx$ then the front projection projects out the $y$-coordinate. For generic knots the front diagram is a self transverse immersion without vertical tangents away from a finite number of semi-cubical cusps. Such a diagram determines the knot completely and it was shown by Ng in \cite{NG1} how to recover a Lagrangian diagram from a front diagram and hence how to compute the DGA.

In Section \ref{diagrams} we introduce the notion of a front diagram for Legendrian knots in $P\times\R$ for $P\ne \R^2$. Unlike the Lagrangian diagram (where we just project to $P$) the classical notion of the front diagram cannot be directly translated (since there is no natural $y$-coordinate to project out). We represent $P$ as a square $D=I^2$ with thin open $1$-handles attached along a distinguished boundary segment $I_0=\{1\}\times I$ at points $p_i^\pm$. We consider an {\emph open handle} to be a thin rectangle which we attach along the short sides.  For Legendrian knots in $D\times \R$ (where the contact form is $dz-ydx$) the front diagram is simply the classical front diagram. By using the Liouville flow to put the "non-affine" Legendrian knots in a standard position we can use this structure to define a notion of a front diagram for these Legendrian knots as well. The front diagrams for Legendrian knots not contained in $D\times \R$ are then defined as the projections of generic Legendrian arcs in $D\times\R$ to $I\times\{0\}\times\R$ with properties as described below. The boundary points of each arc is required to lie in a neighborhood of $p_i^{\pm}\times\R$ obeying certain 1-jet conditions. Furthermore we introduce the notion of a front isotopy in terms of Reidemeister moves , three moves from the classical front diagrams in $\R^3$ and two additional moves corresponding to pushing cusps and double points through the open handles (see Figure \ref{frontmoves23}). Theorem \ref{frontthe} then states that if $K_0$ and $K_1$ are any two Legendrian isotopic knots represented by front diagrams $F_0$ and $F_1$ then there exist a front isotopy from $F_0$ to $F_1$. We show how to recover a Lagrangian diagram from a front diagram and thereby get a combinatorial description of the Legendrian DGA in terms of fronts.

\begin{remark}
While Ng's work makes it possible to directly compute the DGA from the front diagram in the $\R ^2\times \R$ case, we will use the Lagrangian diagram for our calculations. While any given Legendrian knot in $P\times\R$ might need to be put in a standard position (using the Liouville flow) before examining its front diagram, it could immediately be projected to $P$ to get its Lagrangian diagram.     
\end{remark}

The main strength of the front diagram (both in the classical case and in our case) is that it is easy to construct a Legendrian knot with some given front diagram $D$. Using our front diagrams to obtain certain Legendrian knots and using the combinatorial description of the Legendrian DGA in Section \ref{DGAsection} for knots in $P\times\R$ we then establish the following result in Section \ref{exsection}.

\begin{theorem}\label{t:main}
For any $h\in H_1(P\times\R)$ and any positive integer $k$ there exists Legendrian knots $K_1,\dots, K_k$ realizing the homology class $h$ such that $K_i$ and $K_j$ are formally Legendrian isotopic but $K_i$ and $K_j$ are not Legendrian isotopic if $i\ne j$, $i,j\in\{1,\dots,k\}$.  
\end{theorem}

Theorem \ref{t:main} is proved in Section \ref{exsection}. The proof makes use of knots $K$ in classes $h\ne 0$ which have the property that there are no holomorphic disks with one positive puncture and boundary on $K$. (There are no null-homologous knots that satisfy this condition, see Remark \ref{polygonremark}.)

\section{Preliminaries}\label{prel}
In this section we discuss the geometrical and algebraic setup to define contact homology in $P\times\R$.
In Subsection \ref{construct} we show how to give a punctured Riemann surface a contact form $\alpha$ such that it fulfills the demands in Subsection 2.1 in \cite{ekholm1}, allowing contact homology to be defined. In Subsection \ref{prelch} we briefly discuss contact homology. %In Subsection \ref{stablytame} we give the definition of stably tame isomorphisms.

\subsection{Construction of the contact form on $P\times \R$}\label{construct}
Constructing contact manifolds by surgery was studied by Weinstein \cite{WEIN}. We consider, the simplest case, surgery on symplectic surfaces.
A \emph{Liouville vector field} $L$ on a symplectic manifold $M$ is a smooth vector field such that the Lie-derivative of the symplectic form $\omega$ along $L$ is again the symplectic form. Let $D$ be the standard unit disk in $\R^2$ (with coordinates $(x,y)$) defined by $x^2+y^2\leq1$ and let $L$ be the radial symmetric outward pointing Liouville vector field $L:=L(x,y)=(x,y)$. We recall that a Riemann surface is a complex manifold of complex dimension one (and so has real dimension two and a natural complex structure). 
\begin{lemma}\label{pliouville}
Let $P$ be a genus $g$ Riemann surface with boundary which has $p+1$ punctures. Then $P$ admits a  Liouville vector field $L$ with the following properties : $L$ is outwards transverse at the boundary, $L$ has $p+1+2g$ zeros, one with Morse index $0$, and $p+2g$ points with Morse index $1$ for a given Morse function $f$ such that $L$ is gradient like with respect to $f$.
\end{lemma}
\begin{proof}
For the case where $p=g=0$ take the disc $D=P$ with $L(x,y)=(x,y)$ and we are done. Any other (punctured, connected) Riemann surface can be constructed from $D$ by identifying intervals at the boundary. Given two points, $a,b$ on the boundary of $D$ we glue them together, following \cite{WEIN} by locally representing the two pieces as the solution to $F=2y^2-x^2\geq\epsilon=1$ where the boundary is defined by equality. See Figure \ref{glue2} for an illustration of the Liouville vector field before and after the gluing.
\begin{figure}
\includegraphics[scale=0.45]{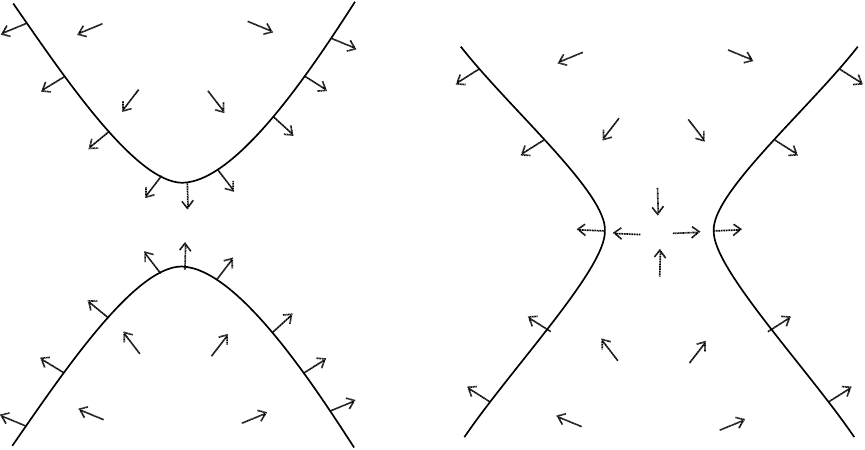}
\caption{The Liouville vectorfield during a handle attachement}
\label{glue2}
\end{figure}

Locally, we can assume that the Liouville vector field $L$ is the gradient vector field of $F$. By letting $\epsilon$ go from $1$ to $-1$, keeping the gradient vector field, we end up with $2y^2-x^2\geq-1$, connecting the two pieces together, keeping $L$ directed outward. This process creates a new critical point of Morse index $1$.  The resulting manifold will have a Liouville vector field $L$ again pointing outwards, with an open handle attaching $a$ to $b$. We continue the attachment process until we have constructed $P$. We need to attach $p+2g$ open handles, resulting in $p+2g$ additional index $-1$ zeros.
\end{proof}
Due to our construction, we consider the interior of our original disc to be the affine part of $P$, and in the same way we consider this affine part times $\R$ to be the affine part of $\P\times\R$ (since it models $\R^3$).

\begin{lemma}\label{preeb}
If $P$ is a $p+1$ punctured Riemann surface of genus $g$ then $P\times\R$ admits a contact form $\alpha$ such that it has finite geometry at infinity according to Definition 2.1 in  \cite{ekholm1}, furthermore, the Reeb vector field will be $\partial_z$ where $z$ is the $\R$ coordinate.
\end{lemma}
\begin{proof}
Construct $P$ with the associated Liouville vector field as described in Lemma \ref{pliouville}. Contract the symplectic form with the Liouville vector field on $P$. This will result in a 1-form $L'$. Let $\alpha=dz-L'$ be the contact form on $P\times \R$. It is straightforward to see that the form has finite geometry at infinity. 
\end{proof}
\begin{remark}
While the definition of finite geometry at infinity is slightly technical, its use is to ensure that no holomorpic discs travel off to infinity. Since the Liouville vector field is expanding, this cannot happen in our case. Each puncture can be compactified to a boundary component with outward pointing Liouville vectorfield. We can attach a cylinder $S^1\times\R_+$ to this puncture, expanding the Liouville vectorfield as fast as necessary.
\end{remark}

\subsection{Contact homology}\label{prelch}
In general the Legendrian contact homology is an algebra over the contact homology algebra of the ambient contact manifold, generated by Reeb orbits. The Reeb vector field on $P\times\R$ does not have any Reeb orbits (since the Reeb vector field is $\partial_z$) and so the Legendrian contact homology is generated by Reeb chords. Contact homology associates a differential graded algebra to a Legendrian sub-manifold $K$ of some contact manifold $M$. The algebra is generated by the chords of the Reeb vector field which ends and starts at $K$. The differential of a chord counts rigid holomorphic discs with punctures on the boundary lying in $P$ such that the boundary is mapped to $\pi(K)$ under certain restrictions. The stable tame isomorphism class of this DGA is an invariant of $K$ up to Legendrian isotopy.
Henceforth, we will assume that $P\times \R$ is given a contact form $\alpha$ as described in Lemma \ref{preeb}. Recall Theorem 1.1 from \cite{ekholm1}, where it is assumed that $P$ has finite geometry at infinity and that $\alpha$ is a contact form expressible as $dz-\theta$ where $\theta$ is a primitive of the symplectic form on $P$. Both conditions are fulfilled for our construction of $P$. This theorem tells us that we have a well-defined contact homology.

\begin{theorem}[Ekholm, Etnyre, Sullivan]
The contact homology of Legendrian sub-manifolds of $(P\times\R,\alpha)$ is
well defined. In particular the stable tame isomorphism class of the DGA associated
to a Legendrian sub-manifold $K$ is independent of the choice of compatible almost
complex structure and is invariant under Legendrian isotopies of $K$.
\end{theorem}

%\subsection{Stably Tame Isomorphisms}\label{stablytame}
%The DGA defined in \cite{ekholm1} is an invariant up to stably tame isomorphism. We recall the definition of stably tame isomorphism.
%\begin{definition}

%\end{definition}

%\begin{definition}
%Given two differential graded algebras $A$ and $B$, we say that $A$ is stably tame isomorphic to $B$ if 
%\end{definition}

\section{A DGA for Legendrian knots in $P\times \R$.}\label{DGAsection}

In this section we give a combinatorial definition of the DGA $A_K$ which is associated to a connected Legendrian sub-manifold $K\subset P\times\R$ based on the geometric definition in \cite{ekholm1}. For a more detailed description of the classical case see Chekanov \cite{CHEK}. We start by defining the graded unital algebra in Subsection \ref{defalgebra} and continue by defining the differential in Subsection \ref{defdifferential}.
In Subsection \ref{linearizedch} we explain how to use linearized contact homology to distinguish DGA's up to stable tame isomorphisms.

\subsection{Defining the Graded Algebra}\label{defalgebra}
Let $K$ be a connected Legendrian sub-manifold in $P\times\R$. Following Subsection $2.2$ in \cite{ekholm1} the algebra of $K$ is generated by chords of the Reeb vector field $\partial_z$. For generic Legendrian knots, the only self-intersections of the projection to $P$ are transverse double points which will correspond to the Reeb chords and hence generate the algebra. Let $A_K:=\Z[a_1,a_2,...,a_n]$ be a unital algebra where $a_i$ are the self-intersections of the Lagrangian diagram of $K$.
This defines the algebra. It is still necessary to define a grading on the algebra. For this we need to define the Maslov index.

\subsubsection{Maslov index}
The Maslov index is described in more detail by Ekholm, Etnyre and Sullivan in Subsection 2.2 of \cite{ekholmr2n} for a more general setting. Given a smooth path $L$ in $P$ such that it forms a continuous loop, it induces a path in the space of Lagrangian subspaces along $L$ (the condition of being Lagrangian is simply being a linear subspace in our case, since the symplectic form is the volume form on $P$). Given a trivialization of the tangent bundle $TP$ along $L$, the loop $L$ induces a path $\tau_L$ in the linear subspaces of $\R^2$, i.e in $\R P^1$. If the path closes up (i.e. $\tau_L(0)=\tau_L(1)$), we let the degree of the induced map $\tau_L:S^1\rightarrow \R P^1$ be the Maslov index $M(L)$. Note that for smooth loops the Maslov index is always even, since the map factors over $S^1$ (that is, the unit vectors in $\R^2$). The Maslov index then depends on a trivialization of $TP$ along the open handles $h_i$ for it to be defined for loops (which close up) of nonzero homotopy in $P$. The consequence of this choice is briefly discussed in Remark \ref{stormaslov}.
\subsubsection{Gradings in $A_k$}
The \emph{degree} $|a|$ of a double point $a$ is defined as follows. The generator $a$ is a double point of the Lagrangian diagram. We can assume that the branches of the Lagrangian diagram are transverse at double points. Then the preimage of $a$ under the projection will consist of two points $a_+$ and $a_-$ with $a_+$ having the higher $z$-coordinate. There are two paths in $K$ directed from $a_+$ to $a_-$, fix one of them as $\gamma$. While $\gamma$ is a loop in $P$, its induced path $\tau_\gamma$ does not close up (due to transversality at $a$), and so we need to attach a small path to $\tau_\gamma$. The attachment is done by simply continouing the path in $\R P^1$ with a positive direction until it closes up. The new path is then called $\tau_\gamma'$.  %The attachement is done as in Figure \ref{loopattach}.
%\begin{figure}
%\includegraphics[scale=0.7]{loopattach}
%\label{loopattach}
%\end{figure}
Then $deg(\tau_\gamma')-1=|a|$. Let $\bar\gamma$ be the other choice. Then $M(K)=\pm M(\gamma-\bar\gamma)=deg(\tau_\gamma')-deg(\tau_{\bar\gamma}')$. We will let the algebra $A_K$ be graded modulo $M(K)$ and thus independent of the choice of $\gamma$. 
\begin{remark}\label{stormaslov}
In some cases (depending on the homology class of $K$, and how $K$ behaves over $P$) the trivializations along $h_i$ can be chosen such that $M(K)$ is zero, allowing a $\Z$ grading. If $K$ is nontrivial in the homology then the trivialization can be always be chosen such that $M(K)$ is arbitrarily large. Later on we will have a grading counted modulo $M(K)$. Letting the $M(K)$ be arbitrarily large allows us to simulate a $\Z-$grading.
\end{remark}
The differential graded algebra for a Legendrian sub-manifold $K$ presented in \cite{ekholm1} has coefficients in the ring $\Z[H_1(K)]$ with the algebra being graded modulo $c(P,\omega)$ where $c(P,\omega)$ is calculated by evaluating twice the first Chern class of $TP$ (equipped with its almost complex structure $J$) on $H_2(P,\omega)$. The Chern class is trivial in our case, initially giving a $\Z$ grading. The generators of $H_1(K)$ are given a grading of their Maslov indexes. In our case, $H_1(K)\simeq\Z$, and is generated by $K$. We calculate the algebra with coefficients in $\Z$ instead of $\Z[H_1(K)]$, giving a grading modulo $M(K)$

%\begin{definition}
%The DGA $A_K$ associated to a Legendrian knot $K$ is a free associative algebra over $\Z$ which is generated by the double points in the Lagrangian diagram of $K$. The algebra is $\Z\slash\Z M(K)$ graded where the \emph{grade} $|a|$ of a generator $a$ is constructed as follows. The generator $a$ is a double point of the Legendrian diagram. We can assume that the Lagrangian diagram is orthogonal at double points. Then the preimage of $a$ under the projection will consist of two points $a_+$ and $a_-$ with $a_+$ having the higher $z$-coordinate. There are two so called ``capping`` paths in $K$ directed from $a_+$ to $a_-$, fix one of them as $\gamma$ and then let $M(\gamma)=|a|$.
%\end{definition}
%This definition mirrors the definitions posed in Subsection 2.2 in \cite{ekholm1}. Reeb chords correspond to double points in the Lagrangian projection to $P$ since the Reeb vector field is $\partial_z$.
%\begin{remark}
%The degree mod $2$ only depends on the local characteristics of the oriented crossing. 
%If we close up the path in the ''smooth`` way, we will get the intersection with the line bundle as twice the intersection of the unit tangent vectors with $V$ while we get an odd number of intersections in the ''non smooth`` case.
%\end{remark}
%\begin{remark}
%In the definition a choice of capping path was made. Let $\gamma'$ be the other choice. Then %M(K)=\pm M(\gamma-\gamma')$ so the degree is independent of choice of capping path.
%\end{remark}

\subsection{Defining the Differential}\label{defdifferential}
We demand that the differential $d$ is linear and obeys the signed Leibniz rule. It is then left to define the differential on generators $a$. 
In Subsection 2.3 in \cite{ekholm1} the differential counts rigid holomorphic discs in $P$ with punctures on the boundary such that the punctures are asymptotic to Reeb chords/double points and the boundary admits a continuous lift to the the Legendrian sub-manifold. In contrast to the more general higher dimension case described in \cite{ekholm1} the count can be reduced to combinatorics due to the Riemann mapping theorem. We will begin by giving a combinatorial definition and later describe the connection to the geometrical definitions in \cite{ekholm1}. The combinatorial description is based on polygons. Let $P_n$ be the convex polygon with $n+1$ corners (we consider $P_0$ and $P_1$ to be "polygons", modeled by a teardrop and a an optical lens respectively) with one distinguished corner. Let the corners be indexed by $\tau_0,...,\tau_n$ ordered by the orientation of the boundary of the polygon, where $\tau_0$ is the distinguished corner. We say that $\tau_0$ is positive and the other marked points are negative. In the Lagrangian diagram we also have corners where the knot self-intersects. We give signs to these corners according to Figure \ref{markedsigns2}.
\begin{definition} 
Let $F_i$ be the set of orientation preserving immersions $f$ of $P_i$ into $P$ (up to orientation preserving diffeomorphisms of $P_i$ mapping $\tau_0$ to $\tau_0$) such that 
\begin{itemize}
 \item $f(\partial P_i)\subset\pi(K)$                                                                                                                                                                                                                                                                                                                                                                                                                                                                                                                                                                                                                                                                                                                          \item $f$ maps $\tau_i$ to double points of $\pi(K)$
 \item $f$ maps neighborhoods of $\tau_i$ to locally convex corners 
 \item $f$ maps the neighborhood of $\tau_0$ into a positive corner 
 \item $f$ maps the neighborhoods of $\tau_i, i>0$ into negative corners                                                                                                                                                                                                                                                                                                                                                                                                                                                                                                                                                                                                                                                                                              \end{itemize}
\end{definition}
Let the product $f(\tau_1)\cdot\dots\cdot f(\tau_i)$ be denoted by $\textbf{b}_f\in A_K$. Note that the product may be empty since the empty word is also considered to be a word.

% Note that the multiplication is well defined since $f(\tau_i)$ is a double point and thus a generator of the algebra.
\begin{definition}\label{differentialdefinition}
Let $K$ be a Lagrangian diagram of some Legendrian knot. Let $a$ be a crossing in $K$. Then $da=\sum_{i=0}^\infty d_ia,$ where $$d_i a=\sum_{f\in F_i:f(\tau_0)=a}w_f\textbf{b}_f$$
\end{definition}
The coefficient $w_f=\pm1$ is a sign associated to the holomorphic disc represented by $f$ which we discuss in \ref{signs}. To avoid the problem of choosing signs of discs, we will restrict ourselves to an algebra over $\Z_2$ in later calculations, allowing us to give the weight $1$ to every disc.

\subsubsection{Polygons and Holomorphic discs}
We recall that a $J$-holomorphic disc in some space $X$ with an (almost) complex structure $J$ is a smooth map from the unit disc $D\subset\C$ into $X$ which is continuous on the boundary and such that the complex structure on $\C$ is (almost) complex linear in the interior. In our case $X=P$ and since $P$ is a Riemann surface, we have a natural complex structure (and thus the map is holomorphic in the interior). We wish to study $0$-dimensional moduli spaces of such discs (up to conformal parametrizations) which are transversally cut out. These are called rigid holomorphic discs and are more closely described in \cite{ekholm1}.
In \cite{ekholm1} the differential $\delta$ of a Reeb chord $a$ counts rigid holomorphic discs with punctures on the boundary such that the boundary of the disc lies on $\pi(K)$ and the punctures on the boundary are sent to double points. Furthermore, we consider one puncture to be a marked point, and require that this puncture is sent to $a$. We can use the Riemann mapping theorem to understand the connection between our immersed convex polygons and rigid holomorphic discs. Given an immersed polygon $P_i$ we can lift the complex structure from $P$ to $P_i$. Then the Riemann mapping theorem gives us a mapping from the unit disc in $\C$ with a marked puncture to $P_i$, giving us a rigid holomorphic disc in $P$. Conversely, any rigid holomorphic disc can be considered an immersed polygon (with convex corners, etc). The correspondence between corners on the boundary of the polygon and punctures on the disc is straightforward. Following \cite{ekholm1} a convex corner in the Lagrangian diagram is said to be positive if traveling along the border of the holomorphic disc takes you from an upper to a lower branch of the knot and negative if it takes you from a lower to an upper branch as depicted in Figure \ref{markedsigns2}. 

\subsubsection{Signs of holomorphic discs}\label{signs}
While computing the algebra $A_K$ over $\Z_2$ gives rise to a useful invariant, it is possible to define the algebra over $\Z$ by giving $w_f$ an appropriate sign. There are several possible ways to assign signs, of which not all are equivalent. Here we describe one combinatorial way to assign signs. Given an immersion $f$ of a polygon $P_i$ let the number of shaded corners in the image of the corners of $P_i$ be $c_f$ following the shading rule in Figure \ref{markedsigns2}. Then $w_f=(-1)^{c_f}$.\\ Signs are assigned to holomorphic discs using a system of coherent orientations, introduced by Etnyre, Ng and Sabloff in \cite{SABLOFF2}. Assigning signs to holomorphic discs is somewhat complicated, and the situation is described by Ekholm, Etnyre and Sullivan in Section 4 in \cite{ekholmoich}. While they consider the contact manifold $\R^{2n}\times\R$ it is possible to use their combinatorial description from 4.5.1 in \cite{ekholmoich} (where $n=1$) to obtain a the combinatorial formula we will use. In \cite{ekholmoich} the combinatorial formula is obtained by comparing certain diagram orientations (see 4.5.4) to coherent orientations induced from basic choices and a spin structure on $K\subset P\times\R$ (see 4.5). We chose the Lie group spin structure on $K$. To relate their combinatorial description to the sign convention described above, choose an immersion of $P\rightarrow\R^2=\C$. Using this immersion we can pull-back diagram orientations and coherent orientations form $\R^2$, allowing us to use the combinatorial descriptions of signs in their paper. Choosing the other spin structure (the null-cobordant spin structure) on $K$ would then change the signs according to Remark 4.35 in \cite{ekholmoich}.

\begin{figure}
\includegraphics[scale=0.5]{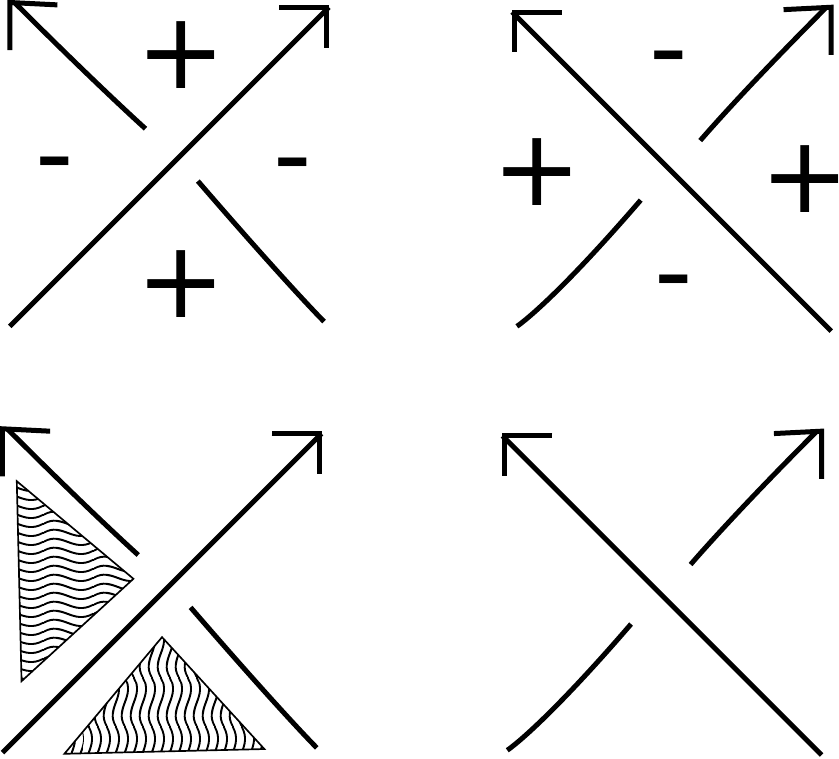}
\caption{The upper crossings show which corners are negative and which are positive. The lower two crossings show the markings assigning signs to the differential.}
\label{markedsigns2}
\end{figure} 

\begin{remark}
Since our construction has followed the construction in \cite{ekholm1} it follows that the differential and the grading are well behaved giving a DGA; for instance, $d^2=0$ and $d$ decreases the degree by one.
\end{remark}

The following theorem holds true for our choice of $P$ as a $p+1$ punctured genus $g$ Riemann surface with contact structure constructed as in Section \ref{prel} and algebras constructed as above.
\begin{theorem}
Let $K,K'$ be two Legendrian knots in $P\times \R$. If $K$ is Legendrian isotopic to $K'$ then there exists a stable tame isomorphism between the associated graded differential algebras $A_K$ and $A_K'$. It then follows that $HC_*(A_{K})$ is isomorphic to $HC_*(A_{K'})$ as graded algebras.
\end{theorem}
\begin{proof}
Since the combinatorial calculation of $A_K$ follows the geometrical definitions in Section 2 of \cite{ekholm1} we can use Proposition 2.6 in their paper stating the Theorem above.
\end{proof}

\subsection{Linearized Contact Homology}\label{linearizedch}
It can be difficult to show that two DGA's are not stable tame isomorphic. One tool to show that two DGA's are not stable tame isomorphic is linearized contact homology, originally introduced by Chekanov in \cite{CHEK}. In linearized contact homology a new differential is constructed acting on the DGA filtered by word length. Let $A$ be a DGA with coefficients from a field (in our case $\Z_2$). We can filter $A$ by word length by letting $A_n$ be the vector space generated by words of length at most $n$ in the generators. We say that a differential is augmented if for each generator $a\in A$ we have that $da$ contains no constant term. An augmented differential acting on $A$ induces a differential $d_1$ on $A_1$ by acting with $d$ on some element in $A_1$ and then projecting back to $A_1$. Since $d_0=0$ the word length never decreases. The linearized contact homology of $A$ is then $ker(d_1)\slash Im(d_1)$. To any such homology we have an associated Chekanov-Poincar\'{e} polynomial $p(\lambda)$  where the coefficient before $\lambda^n$ is the dimension of the $n$-graded components of $ker(d_1)\slash Im(d_1)$ (where the grading is $A$'s original grading). To construct an augmented differential $d'$ from a differential $d$ we let $d'=c\circ d\circ c$ where $c$ is zero on the coefficient ring and maps generators $a_i$ to $a_i+c_i$ where $c_i$ lies in the coefficient ring such that $c_i=0$ if $|a_i|\neq0$. The map $c$ is then called an augmentation if it leads to an augmented differential. The set of Chekanov-Poincar\'{e} polynomials arising from augmentations is an invariant up to stably tame isomorphism.

\section{Knot diagrams}\label{diagrams}
%The study of Legendrian knots in $\R^3$ (with contact form $dz-ydx$) is focused on two diagrams, the front diagram (projection on the $xz$-plane) and the Lagrangian diagram (projection on the $yx$ plane). A front diagram of a Legendrian knot has no vertical tangents but may have cusps. Given a front diagram $F$ with no vertical tangents and with cusps and self-intersections as the only allowed singularities it is possible to recover a Legendrian knot by solving the equation $dz=ydx$.
In this section we describe how to construct appropriate diagrams for Legendrian knots, and how to construct Legendrian knots starting from diagrams.
In Subsection \ref{modelP} we examine $P\times\R$ in more detail and present how to model it as a square with the standard contact structure together with some marked intervals representing where the open handles attach.
We introduce a notion of front diagrams and front diagram isotopy for Legendrian knots in $P\times \R$ in Subsection \ref{frontdiagramsss}. We also show that any pair of knots which are Legendrian isotopic have corresponding front diagrams which are front diagram isotopic. Given a front diagram $F$ we show how to get the Lagrangian diagram $\pi(K)$ for a Legendrian knot $K$ which has $F$ as its front diagram.

\subsection{Modeling $P$ together with its contact form}\label{modelP}
Recall from Section \ref{prel} that the Liouville vector field $L$ was used to construct the Liouville form $L'$ on $P$. Flowing along the Liouville vector field and the Reeb vector field is a Legendrian isotopy. Note that $D\times\R\subset P\times\R$. The contact form $\alpha$ on $P\times \R$ restricts to a contact form on $D\times\R$. We calculate this contact form explicitly. Recall that on the disc we had the Liouville vector field $L:=L(x,y)=(x,y)$. Contracting the symplectic form with $L$ gives $L'=xdy-ydx$ and constructing $\alpha$ gives $\alpha=dz-xdy+ydx$. We construct a contactomorphism from the the standard form $dz-ydx$ by the change of variables
$$x'=(x+y)/2,y'=(x-y)/2,z'=z+xy.$$
Note that this variable change commutes with the projection to $D$. We also that see the Reeb vector field $\partial_{z}=\partial_{z'}$. While this change of variable changes the Liouville vector fields expression in terms of the new coordinates, in the projection to $P$ the vector field $L$ is still repulsive around the origin of $D$. We will henceforth use the standard contact form on $D\times \R$ and denote the coordinates by $x,y,z$. Furthermore, we will model $D$ not as a disc but as a square with the sides parallel to the $x$ and $y$-coordinates.
We model $P$ as this square with some subintervals of the right hand side identified, corresponding to the open handle attachment in the construction of $P$, the identified intervals are then considered to be connected by some open handle. The identification is done respecting the orientation on $\partial D$. We call this the \emph{square model} of $P$. It is well known that any punctured Riemann surface can be obtained by starting with a disc and attaching a number of annuli and a number of punctured tori (that is, tori with a disc removed). Each attached punctured torus corresponds to two open handles, where the second open handle connects two different boundary components and each annulus corresponds to one open handle. For an example with three open handles see Figure \ref{annulitori}. For a $p+1$ punctured, genus $g$ Riemann surface, $p$ annuli and $g$ tori needs to be attached. Since we are allowed to attach the open handles wherever we want, we can choose to construct them along the righthand side and furthermore, attach them in such a way that open handles which come from annuli have positive $y$-coordinates and open handles coming from tori have negative $y$-coordinate, see Figure \ref{annulitori} for an example.
\begin{figure}
\includegraphics[scale=0.35]{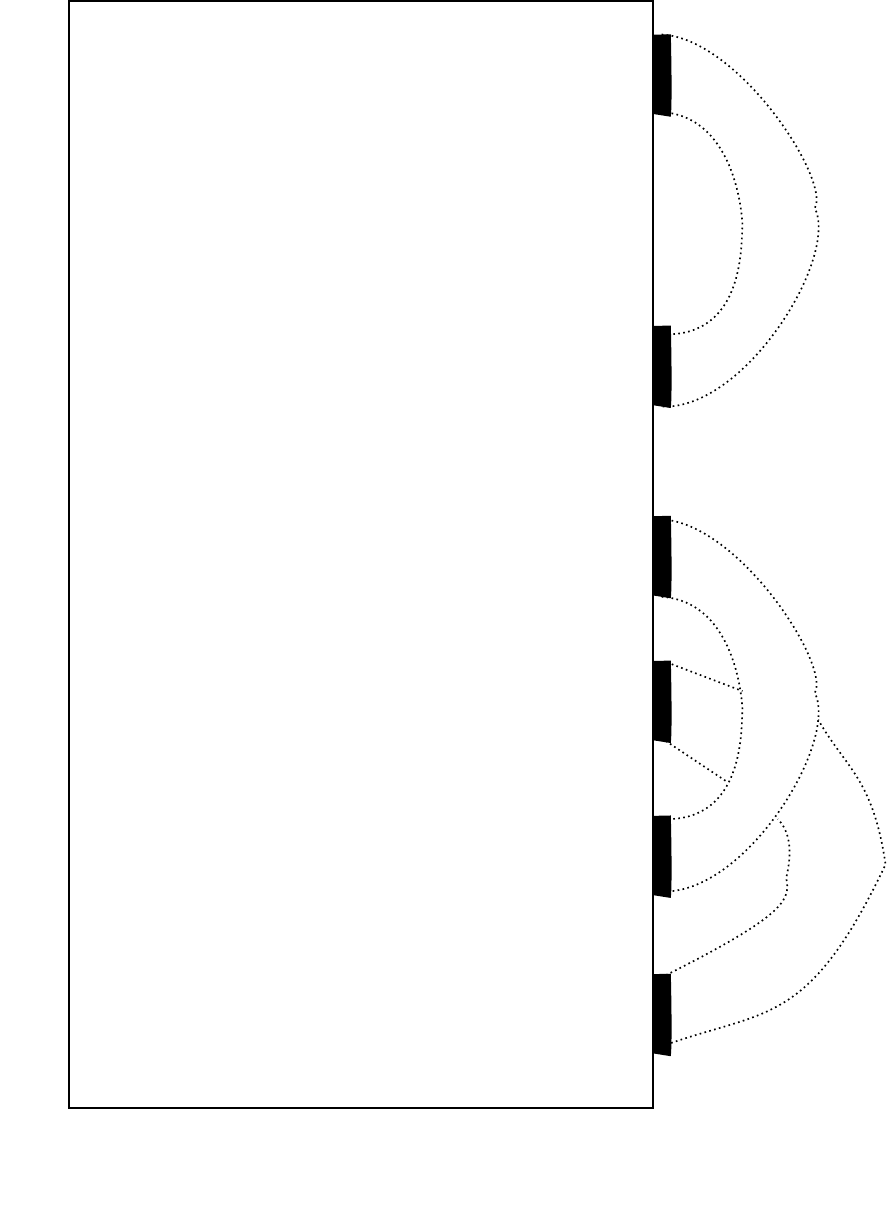}
\caption{A construction of a twice punctured genus one Riemann surface $P$.}
\label{annulitori}
\end{figure}

\subsection{Front diagrams}\label{frontdiagramsss}
%Given a Legendrian knot $K$ we wish to construct a corresponding ``front diagram'' for $K$ that makes sense in our situation. Since the flow of the Liouville vector field is a Legendrian isotopy we can use this flow, preserving the Legendrian isotopy class of $K$. We apply the inverse Liouville  pulling the knot towards the centre of the disc $D$. 
Let $K$ be a Legendrian knot in $P\times \R$. The flow $\phi^t_L$ along the Liouville vector field $L$ (under its natural extension to $P\times \R$) after time $t$ preserves the Legendrian isotopy class of $K$. As $t$ tends to $-\infty$ any compact set tends to some small neighborhood of the union of the stable manifolds. In particular ,since the Liouville vector field $L$ is still attracting towards the line over what was originally the origin of the disk, used to model $P$, after a projection to $P$ the knot will consist of a part, called the \emph{central part}, lying in a small neighborhood $B(0,\epsilon)$ of the origin and several arcs going out through the open handles and then returning to the central part. The height difference in our new coordinates (also called the action) also tends to zero above the central part as shown by the following lemma.

%After flowing long enough the knot will consist of one part $K'$ situated in $B(0,\epsilon)\times \R$ where $B(0,\epsilon)$ is an open ball centered at the origin of $D$ (where the Liouville vector field had a zero) with radius $\epsilon$, and some arcs going through the handles. We call $K'$ the \emph{central part} of the knot $K$.\\
\begin{lemma}\label{litenlada}
Given a Legendrian knot $K$ there exists some function $f(\epsilon)$ such that after flowing along the inverse Liouville flow long enough the part of $(K\cap B(0,\epsilon)\times \R)\subset B(0,\epsilon)\times I_\epsilon$, where the interval $I_\epsilon$ has length $f(\epsilon)$, furthermore $f(\epsilon)$ tends to $0$ as $\epsilon$ tends to $0$. 
\end{lemma}
\begin{proof}
It is easy to see that some function $f(\epsilon)$ satisfying the first part of the theorem exists due to compactness. Assume that we have two points $a,b$ in the central part of $K$ with $z$-coordinates differing by $d>>0$. Since $K$ is connected we can travel from $a$ to $b$ along $K$. The difference in height can come from two sources. It can either come from the central part or from arcs going through the open handles and back, since $K$ can be divided into arcs and the central part. Inside the central part, the height gained corresponds to the area covered in the Lagrangian projection (by integrating $dz=ydx$). However, the ball $B(\epsilon,0)$ has an area tending to zero as $\epsilon$ tends to zero. We visit the central part only finitely many times, leading to a contribution that tends to zero as epsilon tends to zero. We also travel along finitely many handle-arcs so it is enough to show that the height contribution from traveling along one such arc tends to zero. Let this handle-arc be denoted by $A$. The endpoints of $A$ must lie in the central part. We can easily find some Legendrian curve $C$ inside $D\times\R$ such that $C\cup A$ is a Legendrian knot (we simply close up the curve in some way). The height contribution from going along $A$ is then the same as the contribution from going along $C$. After flowing along the inverse Liouville vector field long enough, $C$ will be in the central part and so contribute a height difference bounded by the area of the projection of the central part to $P$ as above. 
\end{proof}

We call a finite collection of curves $C$ in the square, modelling the projecting out the $y$-coordinate in $D\times R$, for a \emph{prefront} if there exists a collection of Legendrian curves $C'$ in $D\times \R$ such that the front projection to the $zx$ plane is $C$. Note that any such collection of Legendrian curves is unique, since the $y$-coordinate of the curves can be regained from the front diagram by solving the equation $dz=ydx$. 

\begin{definition}
We call a prefront $C$ \emph{admissible} for every curve $\gamma\in C$ the endpoints of $\gamma$ lie on marked intervals. Furthermore we require that $\gamma$ is transversal to every marked interval. We also demand that for each associated (by the open handle) pair of intervals $I$ and $I'$  the same number of endpoints lie on both.
\end{definition}
\begin{remark}
Any generic prefront will be obeying the same rules as a standard front diagram in $\R^3$, i.e. no vertical tangents, and the only singularities are semi-cubical cusps and transverse self-intersections.
\end{remark}

Given an admissible prefront $F$ we identify the endpoints of the curves in $F$ in the following manner. For each pair of associated intervals $I$ and $I'$ the endpoints on $I$ are identified with the endpoints on $I'$ in the order which preserves the orientation of the boundary of the square. The curve $\gamma$ obtained in this manner is called the \emph{completed} prefront.

\begin{definition}
Given an admissible prefront $F$ with a completed prefront $\gamma$ we call $F$ a \emph{front diagram} if $\gamma$ is an self transverse immersion of a circle, except for some finite number of points where the image is a semi-cubical cusp. 
\end{definition}

\begin{definition}
Given a generic front diagram $F$ we call the knot diagram $\hat F$ the resolution of $F$ if $\hat F$ is obtained by replacing every crossing, every right and left cusp and every associated interval pair (where the pairing is by the handle) as depicted in Figure \ref{resolution}, noting the different resolutions depending on if such a interval is the upper or lower in the pair.
\end{definition}
\begin{remark}
Note that this resolution for front diagrams which are affine (i.e. never enter the open handles) is the resolution presented by Ng in \cite{NG1}. 
\end{remark}
\begin{figure}
\includegraphics[scale=0.5]{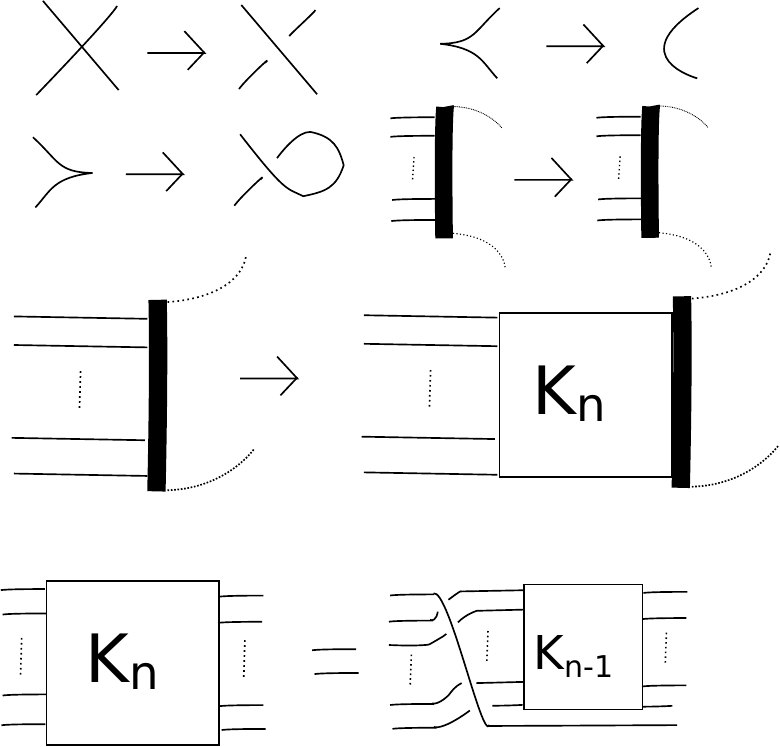}
\caption{The resolution of the front diagram, note the differing resolutions depending on upper and lower handle attachments.}
\label{resolution}
\end{figure} 

We say that two knot diagrams are knot diagram isotopic if there is a isotopy preserving the decorations on the crossings between them.

\begin{theorem}
Given a Legendrian knot $K\subset P\times\R$, it is possible to construct a front diagram $F$ for $K$ such that the resolution $\hat F$ of this diagram is knot diagram isotopic to the Lagrangian diagram of a knot $K'$ which is in the same Legendrian isotopy class as $K$. Furthermore, for any front diagram $F$ there exists a Legendrian knot $K$ such that the Lagrangian diagram $\pi(K)$ is knot diagram isotopic to the resolution of $F$. 
\end{theorem}

\begin{proof}
Choose some small Legendrian perturbation, putting the knot in general position. We begin by flowing along the inverse Liouville vector field $L$ for a long time to put the knot in a standard position.
After this flow the knots will consist of three parts. One part $\hat{K}$ will be inside $B(0,\epsilon)\times\R$. Outside of this piece the knot will consists of a number of arcs (in the rest of the disc and in the handles) which are very close to the solutions of the Liouville flow passing through the zeroes in the handles. These arcs are easy to understand and most of the proof is devoted to examining how they interact with each other and with $\hat{K}$l. See Figure \ref{xyxz} for an example of the Langrangian diagram together with the front diagram for a Legendrian knot where $P$ is the punctured torus and where we have three arcs.

We take care of these pieces separately from the part inside $B(0,\epsilon)\times\R$. Since the restriction of the contact form to $B(0,\epsilon)\times\R$ gives the standard contact form $dz-ydx$ we simply construct the prefront for this pieces of the knot using the standard front projection taken from the classical $\R^3$ situation.\\

We now consider the arcs passing through the handles. Since the arcs were attracted to the solution of the Liouville flow, each arc in itself is very simple. It starts at $\hat{K}$ and travels outward from the disc to a handle, goes through the handle, and returns to $\hat{K}$ without any self intersections in the Lagrangian diagram (all such intersections are pushed to $\hat{K}$ by the inverse Liouville flow.
We wish to understand how these arcs interact with each other and how they enter the $\hat{K}$ part of the knot. Two arcs passing through different handles will not have any intersections in the Lagrangian diagram due to the flowing separating them. Take two arcs $A$ and $A'$ passing through the same handle. Since the Liouville flow pulled the knot towards the curve passing through the zero of the Liouville vector field in this handle, the two arcs will be very close in the Lagrangian projection. See Figure \ref{xyxz} for an example of the Lagrangian projection together with the front diagram. 
\begin{figure}
\includegraphics[scale=0.2]{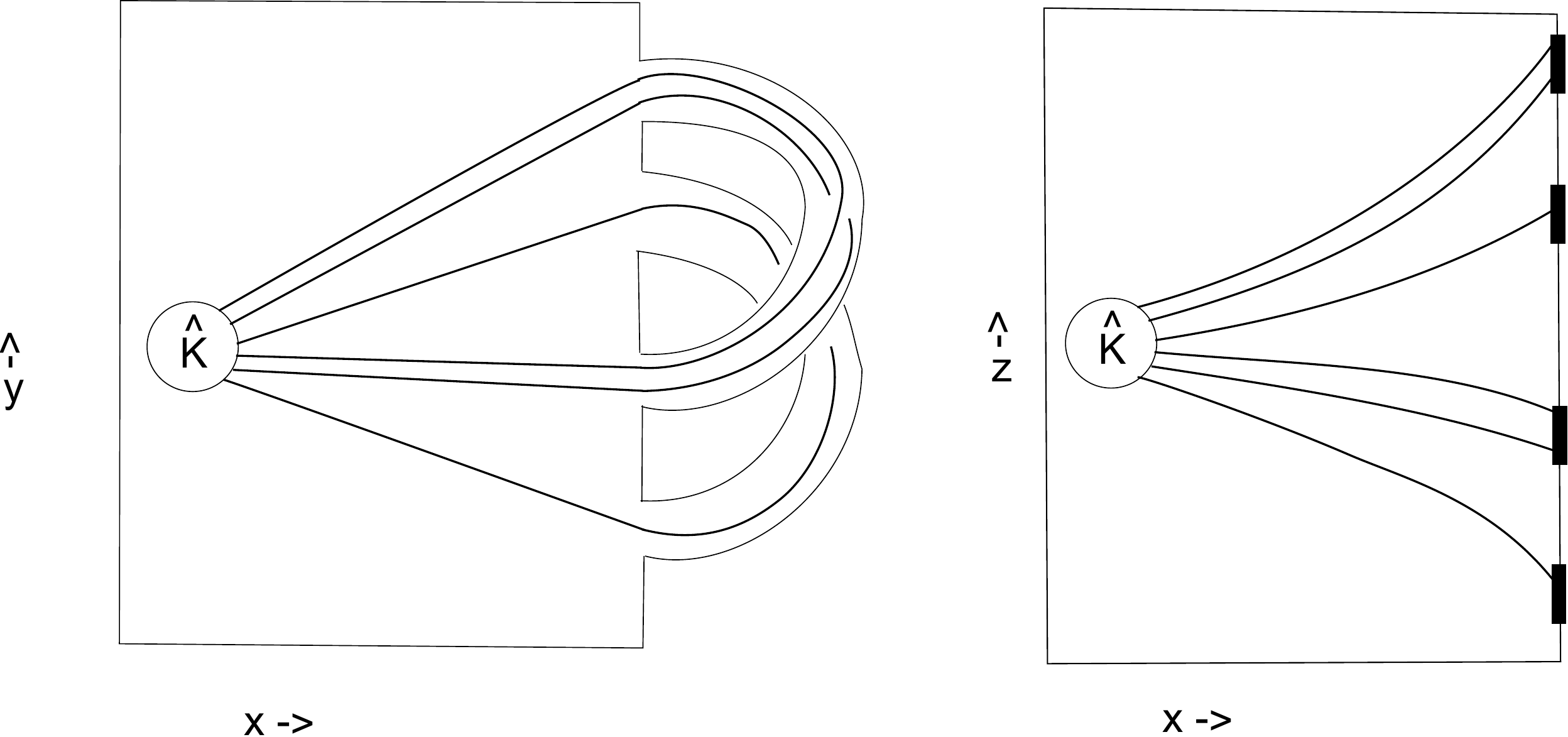}
\caption{A Lagrangian diagram together with the front diagram. Note that the handles are not included in the front diagram since we can only mark when arcs enter the handles, but not the handles themselves.}
\label{xyxz}
\end{figure} 

Since the arcs are very close, the value of the Liouville form integrated along the arcs will be very close. Since $dz=L'$ for Legendrian knots the height differences achieved by passing through the handle must be very close for the two arcs as well (and can be made arbitrarily close by flowing long enough) and hence the sign of the difference between their respective $z$-coordinates does not change when passing through the handle. Thus the $z$-order in which they enter $\hat{K}$ does not change either.
In the Lagrangian projection to $D$ inside the square model of $P$, outside of the central part $B(0,\epsilon)$ arcs will be close to straight lines going from the intervals denoting the open handle attachment to $B(0,\epsilon)$. We can assume that these straight lines do not intersect (by first choosing a small $\epsilon$ making any eventual intersections very close to $B(0,\epsilon)$ and then choosing a slightly larger $\epsilon$ absorbing any intersections. 
Since they do not intersect and go from the intervals denoting the open handle attachments to the center and the open handle attachments all have differing $y$-coordinates, the arcs will be ordered by the $y-$coordinate of their entry to the intervals. Furthermore, this ordering will be preserved for any $x>\epsilon$, and arcs entering the same interval will have $y-$coordinates very close to each other. 
In the $xz$ projection in $D\times\R$, the ordering of the $y$-coordinates will correspond to a similar ordering of the slopes of the arcs by solving the equation $dz=ydx$, with higher $y-$coordinate corresponding to a higher slope. By Lemma \ref{litenlada}, the arcs start inside some small cube above the center of $D$. Since the arcs travel outwards from $B(0,\epsilon)\times I_\epsilon$ with ordered slopes, by making $I_\epsilon$ of small enough length, the slope order will correspond to the $z$-coordinate order when entering the handles. This justifies drawing intervals at the appropriate $z-$coordinates corresponding to the handles which the arcs are passing through. Since the height difference between two arcs passing through the same handle did not change significantly, the order of passing into a handle, ordered by $z$-coordinate and going out the other end of the handle, ordered by $z-$coordinate, must be preserved. 
Going back through the equivalent orderings, also the $y$-coordinate orderings must be preserved. Thus, two arcs passing through the same handle must cross at least once in the Lagrangian diagram, they cannot cross more than once due to the flow. Since we know their relative heights outside of the handle, and since any height-changes inside the handles are very close (being the integration of the Liouville form), we know which arc passes above the other. Since we are only interested in the Lagrangian diagram up to knot diagram isotopies, we can move out these crossings outside the handle. This corresponds to the resolution of the handles. We know that we have no crossings in the Lagrangian diagram outside of the handles and $B(0,\epsilon)$. Inside $B(0,\epsilon)$ we can use Ng's \cite{NG1} techniques to retrieve the Lagrangian diagram.  This gives a complete picture of the Lagrangian diagram up to knot diagram isotopy.\\ 

Given a front diagram $F$ we construct an associated Legendrian knot as follows. Choose some Legendrian arcs passing through the handles where the curve in the front diagram enters. After flowing along the inverse Liouville flow for long enough time, these arcs will stay Legendrian and their endpoints will stick into $D\times\R$. In $D\times \R$ we can attach the pieces of the diagram entering the handles to the arcs lifted to the front diagram in $D\times\R$. Inside the square we can simply solve the equation $dz=ydx$ to get $y-$coordinates making the result Legendrian. Outside of the square we will have our (near identical) arcs which we knew were Legendrian. The front diagram for this Legendrian knot is then $F$.
\end{proof}
\begin{remark}
Since the resolution of a Legendrian knot $K$ discovers every crossing in the Lagrangian diagram, it discovers every Reeb chord and can thus be used to calculate the DGA. Furthermore, since any front diagram corresponds to a Legendrian knot, we have a method for constructing examples of Legendrian knots in $P\times \R$.
\end{remark}
%\begin{definition}
% Given two front diagrams $F$ and $F'$ we say that $F$ is front diagram isotopic to $F'$ if $F$ can be deformed to $F'$ using the moves depicted in figure XXX.
%\end{definition}
%\begin{theorem}
%Two Legendrian knots $L$ and $L'$ are Legendrian isotopic if and only if their front diagrams are front diagram isotopic.
%\end{theorem}
%\begin{proof}
%Given the knot $L$ and $L'$
%\end{proof}
\begin{definition}
We say that two front diagrams are \emph{equivalent} if there exists a sequence of moves taking one of them to the other. The allowed moves are depicted in Figure \ref{frontmoves23}. 
\end{definition}

\begin{figure}
\includegraphics[scale=0.5]{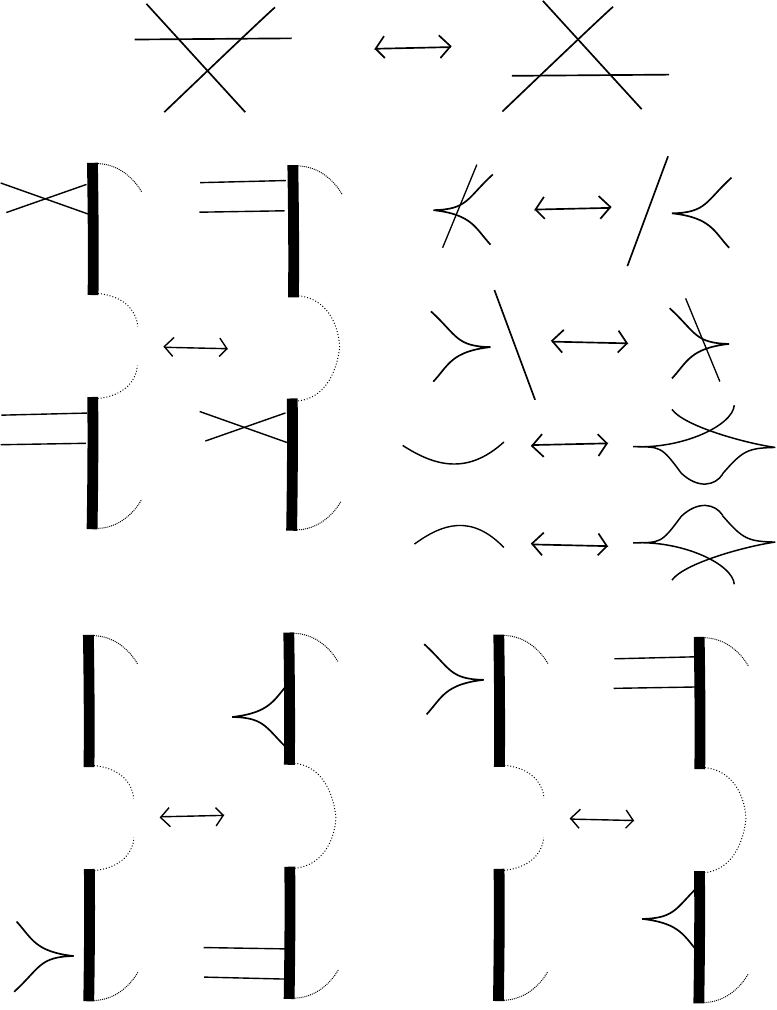}
\caption{Moves defining front diagram equivalence.}
\label{frontmoves23}
\end{figure}

\begin{theorem}\label{frontthe}
If two Legendrian knots $K$ and $K'$ are Legendrian isotopic the front diagrams associated to $K$ and $K'$ will be equivalent.
\end{theorem}
\begin{proof}
Let $f(t)$ be the Legendrian isotopy taking $K$ to $K'$. We can act on the entire isotopy by flowing along the inverse Liouville vector field. As long as all the knots are in general position with regards to the handle-part of the unstable solution passing through the zeros of the Liouville flow in the handles the entire isotopy will be a Legendrian isotopy of curves in $D\times\R$ and can thus be reduced to a sequence of classical front diagram moves (those not involving handles in Figure \ref{frontmoves23}). If the knot ever is in a non-general position with regards to the unstable solution, we either have a tangency to the unstable solution or a double point (that is, we have a loop or a crossing passing through a handle) in the generic case. Since the isotopy can be assumed to be generic, there are only finitely many such occurrences. In each case, we examine the diagram just before the collision with the unstable solution, flow against the Liouville vector field, and compare it to the picture after the collision, again using the inverse Liouville flow. These situations correspond to the moves involving the handles in Figure \ref{frontmoves23}

\end{proof}

\section{Examples}\label{exsection}
In this section we consider some examples of Legendrian knots. Any affine Legendrian knot in $\R^3$ is naturally included as a Legendrian knot in $P\times \R$. We calculate the DGA for several non-affine Legendrian knots and show that Legendrian knots which are formally Legendrian isotopic but not Legendrian isotopic exists in every homology class in $P$. We will use Chekanovs knots to construct non-affine knots which need the DGA to distinguish them.
An interesting collection of knots realizing every nontrivial homology class in $P\times \R$ is given by the following lemma.
\begin{lemma}\label{KH}
For each nontrivial homology class $h\in H_1(P\times \R)$ there exists a Legendrian knot $K_h\subset P\times\R$ such that $K_h$ realizes the homology class $h$ and such that every polygon immersed as in Subsection \ref{defdifferential} has at least two positive punctures.
\end{lemma}
\begin{proof}
Let $P$ be a Riemann surface with $p+1$ punctures and genus $g$. We model $P$ with our standard square model, where we have the open handles arising from attaching a punctured torus has positive $y$-coordinates and the open handles arising from attaching annuli with negative $y$-coordinate. Then $H_1(P\times \R)=\Z^p\times(\Z\times\Z)^g$. We construct the Legendrian knot $K_h$ by constructing Legendrian knots $K_{a,b}$ for the genus contributing pieces and Legendrian knots $K_c$ for the puncture contributing pieces such that they admit no immersed polygons with less than two positive corners realizing the corresponding homology classes $(a,b)$ respectively $c$. We begin with the case $(a,b)$. 
Let $k_1,k_2,\tau$ be an integer solution to the vector-equation $k_1(1,1)+k_2(1,-1)+\tau(1,0)=(a,b)$ such that $\tau$ is either $0$ or $1$. Take $k_1$ copies of curve $A$, $k_2$ copies of curve $B$ and $\tau$ copies of curve $C$ where the curves are as depicted in Figure \ref{ABC}. Taking a negative number of curves is interpreted as switching orientation of the curve. Note that the curves do not intersect either copies of themselves or each other (except for $C$ which could intersect itself, hence we only use one copy of it) before the resolution.
\begin{figure}
\includegraphics[scale=0.2]{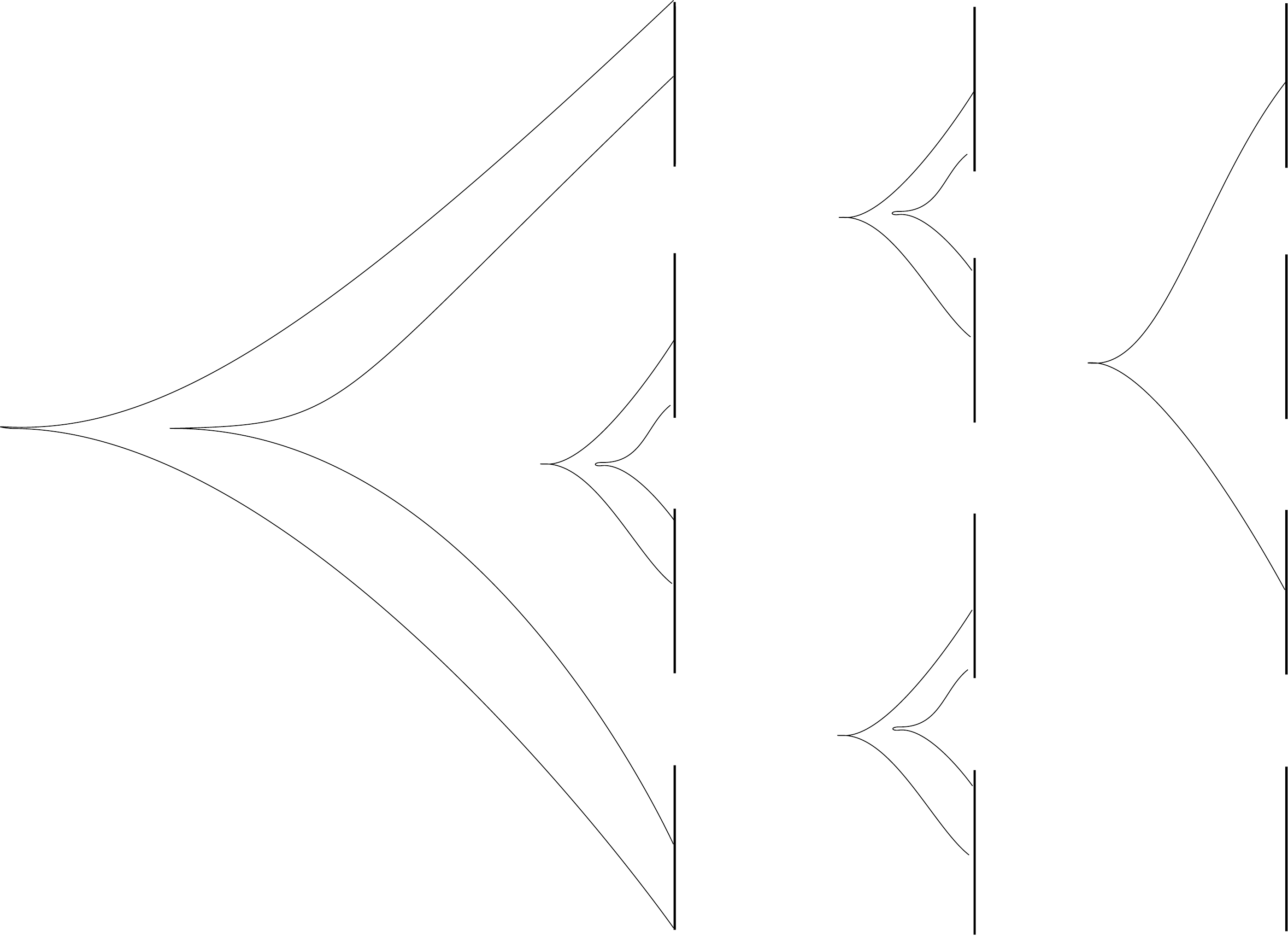}
\caption{Two copies of curve A, two copies of curve B and one copy of curve C.}
\label{ABC}
\end{figure}
Note that this collection of curves do not cross each other and have no righthand cusps. To make just one knot out of all these components we attach ''horizontal braids`` $T_n$ inductively defined as in Figure \ref{TN} to the top interval for the A-curves and the B-curves (the C-curve is already one component). The A,B and C parts are then connected together by adding replacing parallel horizontal pieces with one part from two different connected components with $T_2$ braids where the orientations agree. We call the resulting knot $K_{a,b}$.
In the case of a puncture encapsulating open handle with homology class $c$, we construct the knot directly as seen in Figure \ref{knotc}, with $\lfloor c/2\rfloor$ denoting the integer part of $c/2$. We can assume that $c>0$, since in the case of $c<0$ we follow the construction for $-c$ and switch orientation.

\begin{figure}
\includegraphics[scale=0.7]{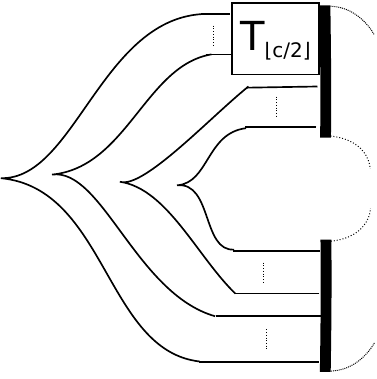}
\caption{The front diagram realizing $c$, there are $c$ strands passing through the handle.}
\label{knotc}
\end{figure}

After resolving these diagrams we will get a number of crossings from the $T_n$'s and some crossings from the handle resolution.
However, no immersed polygon has less than two positive punctures, thus the DGA is just the free algebra of the generators $a_i$ with $da_i=0$. This can be seen by examining the horizontal braids from which crossings originate, both coming from the handle resolution and the $T_n$ constructions. In any such horizontal braid it is easy to see that any immersed polygon is forced to either extends through the entire braid (perhaps picking up some negative corners) or start at a positive corner and extends out from one side. Since the horizontal braids are connected to each other, any attempt to immerse a polygon can only start and end at a positive corner, any immersed polygon must have at least two positive punctures.
Recall that we could choose the trivialization along the handles. This allows us to choose a trivialization giving us an arbitrary high Maslov index for this knot.
To construct the knot $K_h$ with homology $(c_1,c_2,...,c_p,a_1,b_1,a_2,b_2,...,a_g,b_g)$ we simply take the knots $K_{c_1},..K_{c_p},K_{a_1,b_1},...,K_{a_g,b_g}$ and glue them together as shown in Figure \ref{glue}, depending on how the orientations agree. In the case of one of the terms being zero, we just ignore this knot. Since we assumed that the homology was nontrivial, we will end up with at least one piece. These gluings will obviously not allow any more polygons with less than two positive punctures, and so the resulting knot $K_h$ will fulfill the demands in the theorem.
\begin{figure}
\includegraphics[scale=0.3]{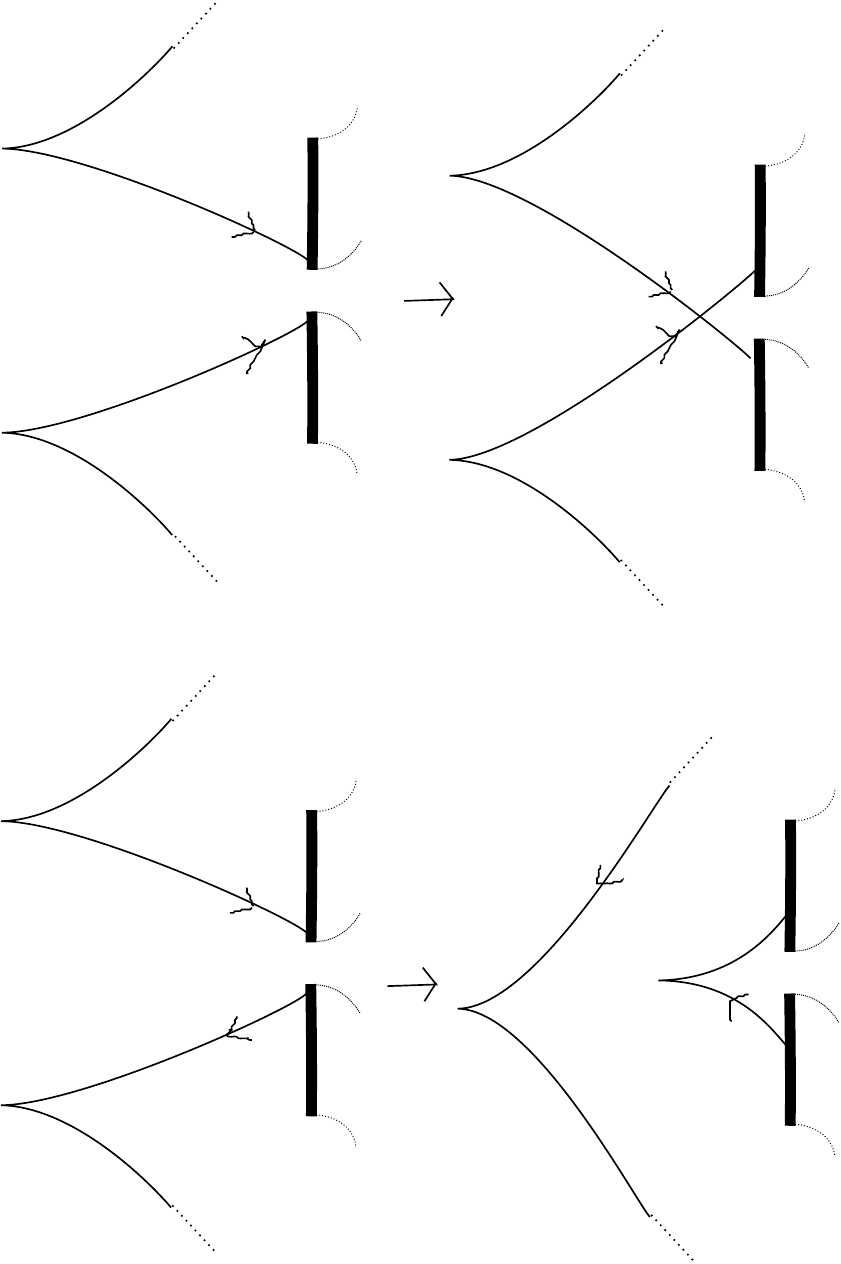}
\caption{We glue together different knots as depicted, dependent on the orientations.}
\label{glue}
\end{figure}

\begin{figure}
\includegraphics[scale=0.3]{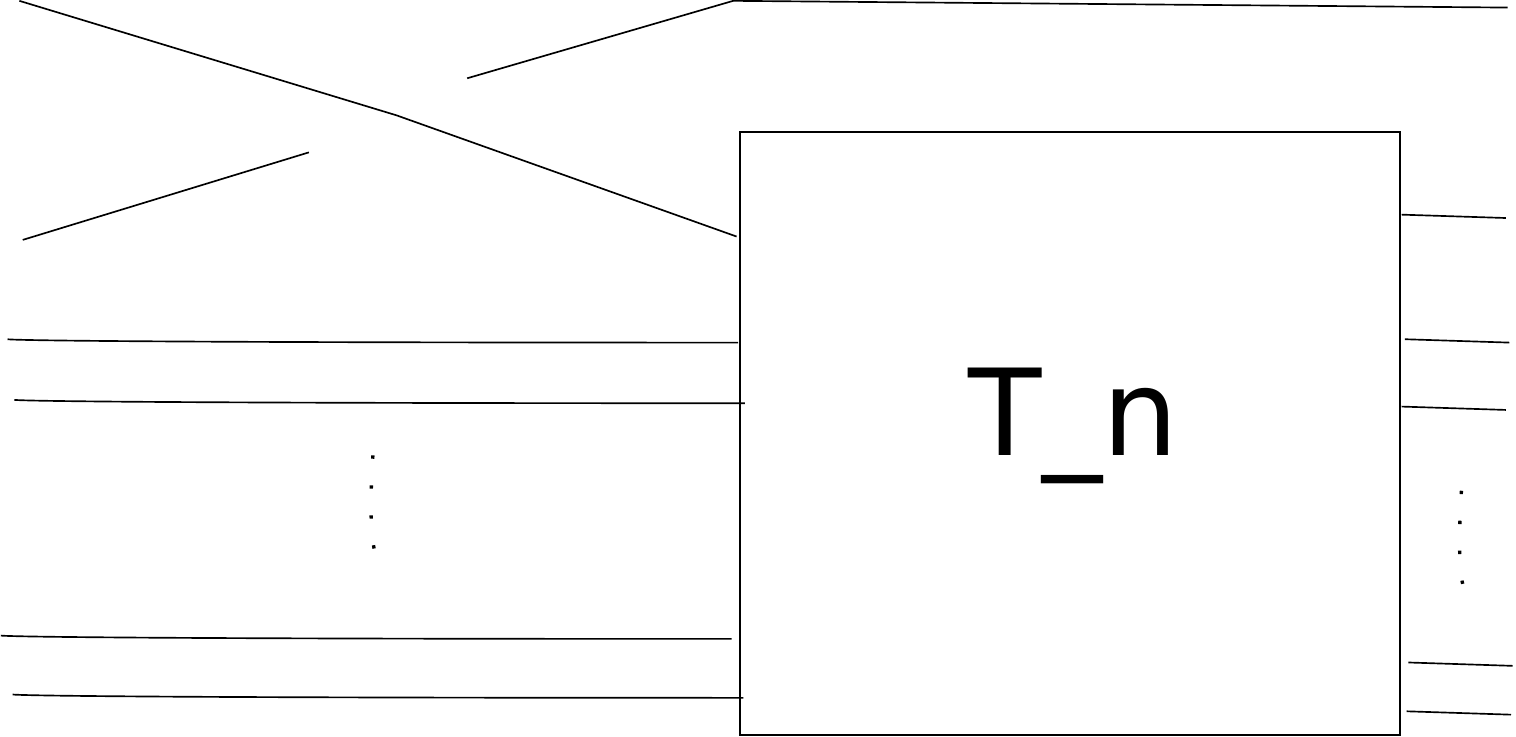}
\caption{We construct $T_{n+1}$ on $n+1$ strands using $T_n$}
\label{TN}
\end{figure}

\end{proof}

\begin{remark}\label{polygonremark}
Using results of Bourgeois, Ekholm and Eliashberg \cite{beeeffect}\cite{beesymp} one can show that every Legendrian knot in $\R^3$ has some immersed polygon in its Lagrangian diagram with exactly one positive puncture [Ekholm, personal communication].
\end{remark}

This collection of knots enables us to relatively easily attach other knots to it, without destroying too much of the homology. This gives rise to the following theorems.

\begin{theorem}\label{homologyknots}
For each homology class $h$ in $H_1(P\times R)$, there exists 2 Legendrian knots $K$ and $K'$ which cannot be distinguished by the classical invariants yet have different contact homology and so are not Legendrian simple.
\end{theorem}
\begin{proof}
In the trivial case we can just take Chekanovs knots, presented in \cite{CHEK} and use an affine diagram, since the calculation of the contact homology coincides with the classical calculation, they are not isotopic. In the nontrivial case we let the sought after homology class be $(a,b)$.
We attach the Chekanov knots $L$, $L'$ which are formally Legendrian isotopic to $K_{a,b}$ as depicted in figure \ref{lkab}.
\begin{figure}
\includegraphics[scale=0.19]{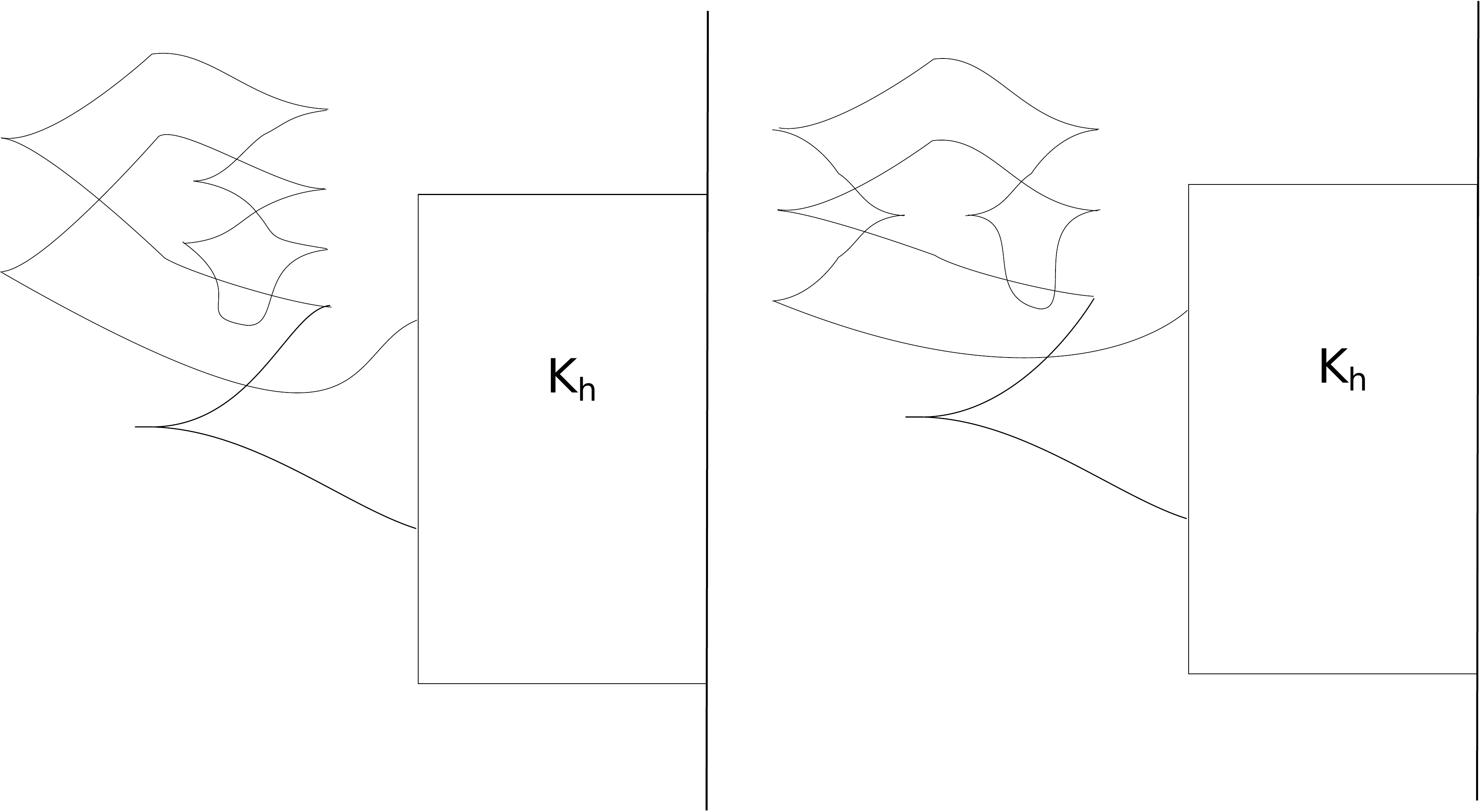}
\caption{The attachment of Chekanovs two knots to $K_h$}
\label{lkab}
\end{figure}

The resulting knots are isotopic and of the same Maslov index. We can see that since the Chekanov knots are formally Legendrian isotopic in $\R^3$ and we just attach our knot at one point, we can fix them at this point and construct this formal Legendrian isotopy inside a sphere outside the point (in the non-affine part, they are identical so nothing needs to be done). We can calculate the linearized contact homology for our new knots $K$ and $K'$ with coefficients in $\Z_2$ (to simplify the calculations) after resolving the front diagram. We separate the generators into three different kinds. We denote the generators originating from the knots $L,L'$ by $\alpha_1,...,\alpha_9$ and $\alpha'_1,...,\alpha'_9$ respectively. The crossings originating from $K_h$ are called $\beta_1,...,\beta_n$ and $\beta'_1,...,\beta'_n$ respectively. Note that $|\beta_i|=|\beta'_i|$. The crossing originating from the attachments we call $\gamma$ and $\gamma'$.
The part of the resolved knots coming from $K$ and $K'$ (which is the interesting part from the point of view of contact homology calculations) are depicted in Figure \ref{chekanovres}.
\begin{figure}
\includegraphics[scale=0.28]{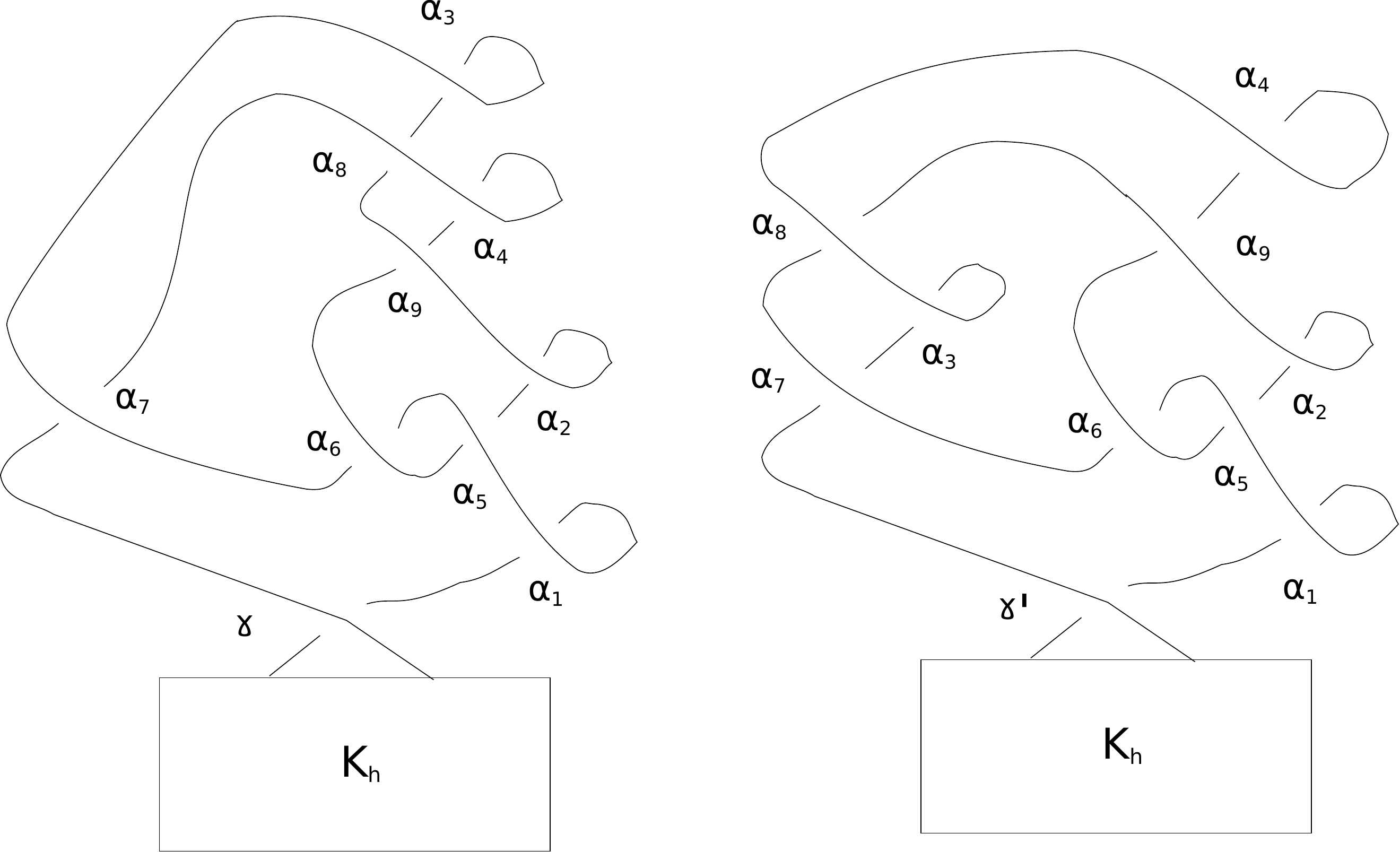}
\caption{The knots $K$ and $K'$ after attachment and resolution.}
\label{chekanovres}
\end{figure}

We can easily see that $|\gamma|=|\gamma'|=m(L)=m(L')=0$ and (following the capping paths that do not visit the handles) that 
$$|\alpha_i|=|\alpha'_i|=1,  i=1...4,$$
$$|\alpha_5|=2, |\alpha'_5|=0$$
$$|\alpha_6|=-2,|\alpha'_6|=0 $$
$$|\alpha_i|=|\alpha'_i|=0,  i=7...9.$$
While the degree is counted modulo the Maslov index of the entire knot, we can assume that this number is arbitrarily large by choosing an appropriate trivialization along the handles. This might change the degrees of the $\beta$ and $\beta'$ crossings, but will not change the degree of the other crossings since they can be computed without going through arcs. Since the grading of the algebra is done modulo something arbitrarily large the degrees $-2,2,0,1$ are distinct.
The differentials are then as follows for the $K$ knot.
$$d(\alpha_1)=1+\alpha_7\gamma+\alpha_5\alpha_6\alpha_7\gamma$$
$$d(\alpha_2)=1+\alpha_9+\alpha_9\alpha_6\alpha_5$$
$$d(\alpha_3)=1+\alpha_7\alpha_8$$
$$d(\alpha_4)=1+\alpha_8\alpha_9$$
$$d(\alpha_i)=d(\gamma)=d(\beta_j)=0, i\geq4.$$
We recall that an augmentation is an algebra homomorphism $c$ sending $\alpha_i$ to $\alpha_i+c_i$, $\beta_j$ to $\beta_j+b_j$ and $\gamma$ to $\gamma +g_1$ such that $c\circ d$ lacks constant terms and such that it acts as identity on generators of non-zero degree. Since the $\beta_j$ generators never appear in the differentials, any choice of $b_j$ (respecting the degree condition) works. We need to find solutions to a corresponding system of equations for the generators on which the differential acts non-trivially to find possible augmentations. Since we are working over $\Z_2$ this is not very difficult.
$$0=1+c_7g_1+c_5c_6c_7g_1$$
$$0=1+c_9+c_9c_6c_5$$
$$0=1+c_7c_8$$
$$0=1+c_8c_9$$
It is obvious that the only possibility which gives an augmentation is $c_7=c_8=c_9=g_1=1.$ with $c_i$ associated to $a_i$ and $g_1$ associated to $\gamma$.  
On the linearized level we have that 
$$d^1(\alpha_1)=\alpha_7$$
$$d^1(\alpha_2)=\alpha_9$$
$$d^1(\alpha_3)=\alpha_8+\alpha_7$$
$$d^1(\alpha_4)=\alpha_9+\alpha_8$$
It is easy to see that the homology is generated by $\alpha_i$ for $i=1...6$, $\gamma$ and the $\beta_i$'s.
Then the associated polynomial is $\lambda^2+4\lambda+\lambda^{-2}+p(\lambda)$ where 
$$p(\lambda)=\sum_{j=-\infty}^{\infty}\#\{i:|\beta_i|=j\}\lambda^j.$$
In the $K'$ case, we have the following differentials:
$$d(\alpha'_1)=1+\alpha'_7\gamma+\alpha'_5\alpha'_6\alpha'_7\gamma$$
$$d(\alpha'_2)=1+\alpha'_9+\alpha'_9\alpha'_6\alpha'_5$$
$$d(\alpha'_3)=1+\alpha'_8\alpha'_7$$
$$d(\alpha'_4)=1+\alpha'_8\alpha'_9$$
$$d(\alpha'_i)=d(\gamma')=d(\beta'_j)=0, i\geq4.$$
While the equations are the same as for $K$ the grading differs which gives us other possibilities for augmentations (and, as we will see, is the crucial difference between the knots).
The augmentations must again be $c'_7=c'_8=c'_9=g'_1=1$ with the added demand that while $c'_5c'_6=0$ one of them might be nonzero (which was earlier forbidden by the degree condition). Recall that $d^1$ is the linear term of the augmented differential $c\circ d\circ c$.
On the linearized level we have that 
$$d^1(\alpha_1)=\alpha'_7+c'_5\alpha'_6+c'_6\alpha'_5$$
$$d^1(\alpha_2)=\alpha'_9+c'_5\alpha'_6+c'_6\alpha'_5$$
$$d^1(\alpha_3)=\alpha'_8+\alpha'_7$$
$$d^1(\alpha_4)=\alpha'_9+\alpha'_8$$
The homology is then (again) generated by $\alpha'_i$ for $i=1...6$ and the $\beta'_i$ (for any solution of $c'_5c'_6=0$). Since $|\beta_i|=|\beta'_i|$ we have the associated polynomial as $2+4\lambda+p(\lambda)\neq\lambda^2+4\lambda+\lambda^{-2}+p(\lambda)$. Since we examined all possible augmentations in both cases, the sets of Chekanov-Poincar\'{e} polynomials associated to the knots both consist of a single element. Thus, the two knots are not Legendrian isotopic since it is enough to compare these elements.
\end{proof}
This theorem can be strengthened further using the constructions in the proof to prove Theorem \ref{t:main}, here restated.
\begin{theorem*}
For any $h\in H_1(P\times\R)$ and any positive integer $k$ there exists Legendrian knots $K_1,\dots, K_k$ realizing the homology class $h$ such that $K_i$ and $K_j$ are formally Legendrian isotopic but $K_i$ and $K_j$ are not Legendrian isotopic if $i\ne j$, $i,j\in\{1,\dots,k\}$. 
\end{theorem*}
\begin{proof}
We use a similar construction as in the proof of Theorem \ref{homologyknots}. Instead of taking one copy of Chekanov's knot $L$ (or $L'$) and attaching it to our knot $K_h$, we attach $i$ (with $0<i\leq k$) copies of $L$ and $k-1-i$ copies of $L'$ as depicted in Figure \ref{chekanovattach}. 
\begin{figure}
\includegraphics[scale=0.2]{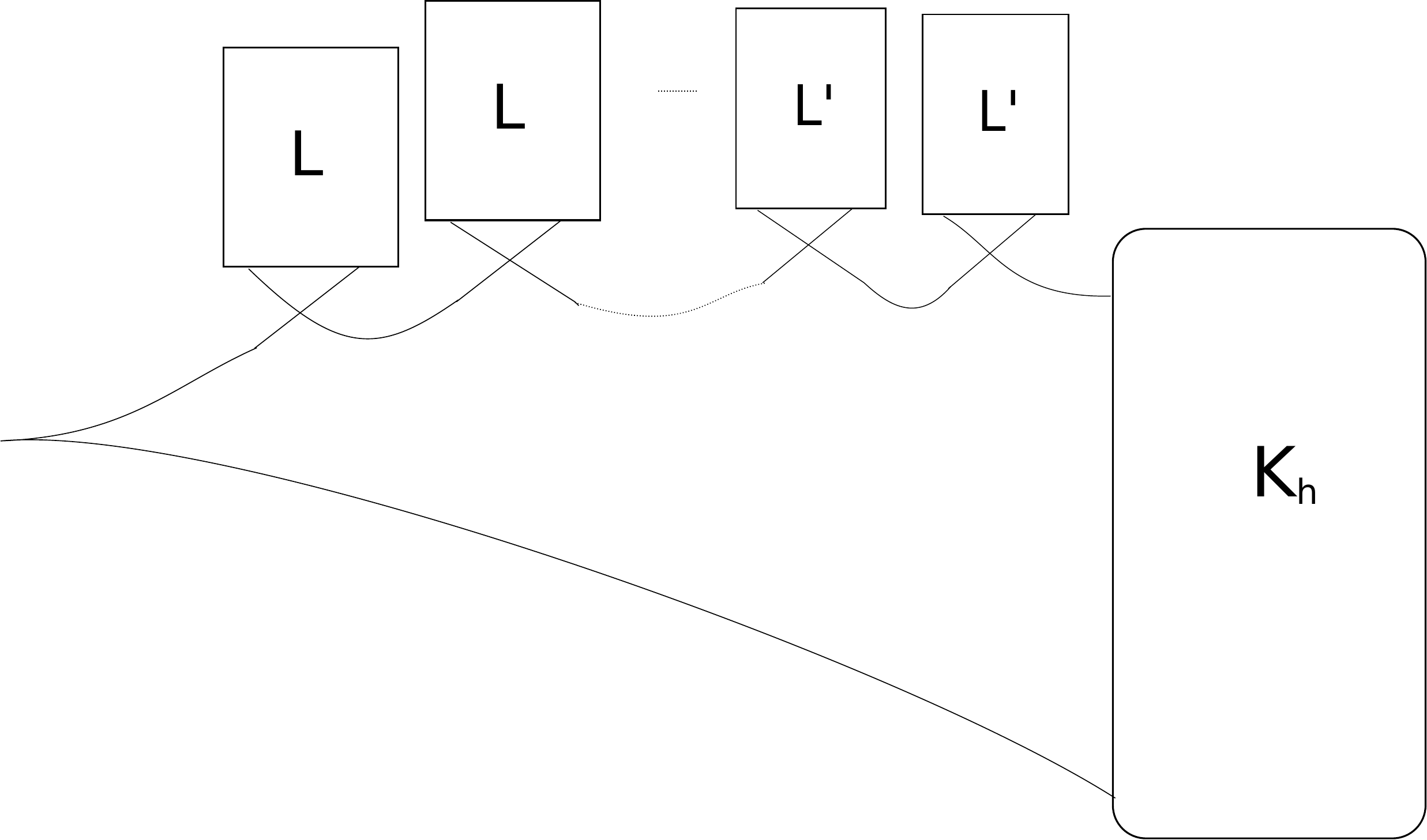}
\caption{Attaching many copies of Chekanovs knots.}
\label{chekanovattach}
\end{figure}
For each choice of $i$ we get a knot $K_i$, in total we get $k$ knots.
Clearly the $k$ knots will be formally Legendrian isotopic, by the same argument as before, after the gluing. We again use the Legendrian contact homology and linearize it. The copies of $L$ and $L'$ do not interact with each other, so we will get $n$ copies of the equations for the $K$ case and $k-1-n$ copies of the equations for the $K'$ case together with the $\beta$ equations. The associated polynomials will be $i\lambda^2+4\lambda+2k-2-2i+i\lambda^{-2}+p(\lambda)$. The polynomial $p(\lambda)$ is independent of $i$ (and is indeed the same polynomial as in the proof of Theorem \ref{homologyknots}). Hence, we have $k$ Legendrian knots which are not pairwise Legendrian isotopic by letting $i$ range from $0$ to $k-1$.
\end{proof}

\begin{example}
In Figure \ref{nollhomolog} we give an example of a knot which is zero in homology but not homotopically trivial in the case of $P$ being the punctured torus. Furthermore, this knot has a relatively simple contact homology, with four generators $a_1,...,a_4$ such that $da_2=da_4=0, da_1=a_2,da_3=a_4$. Using this knot we can again construct arbitrarily high numbers of pairwise non Legendrian isotopic knots in its homotopy class using the techniques above.
\end{example}
\begin{figure}
\includegraphics[scale=0.2]{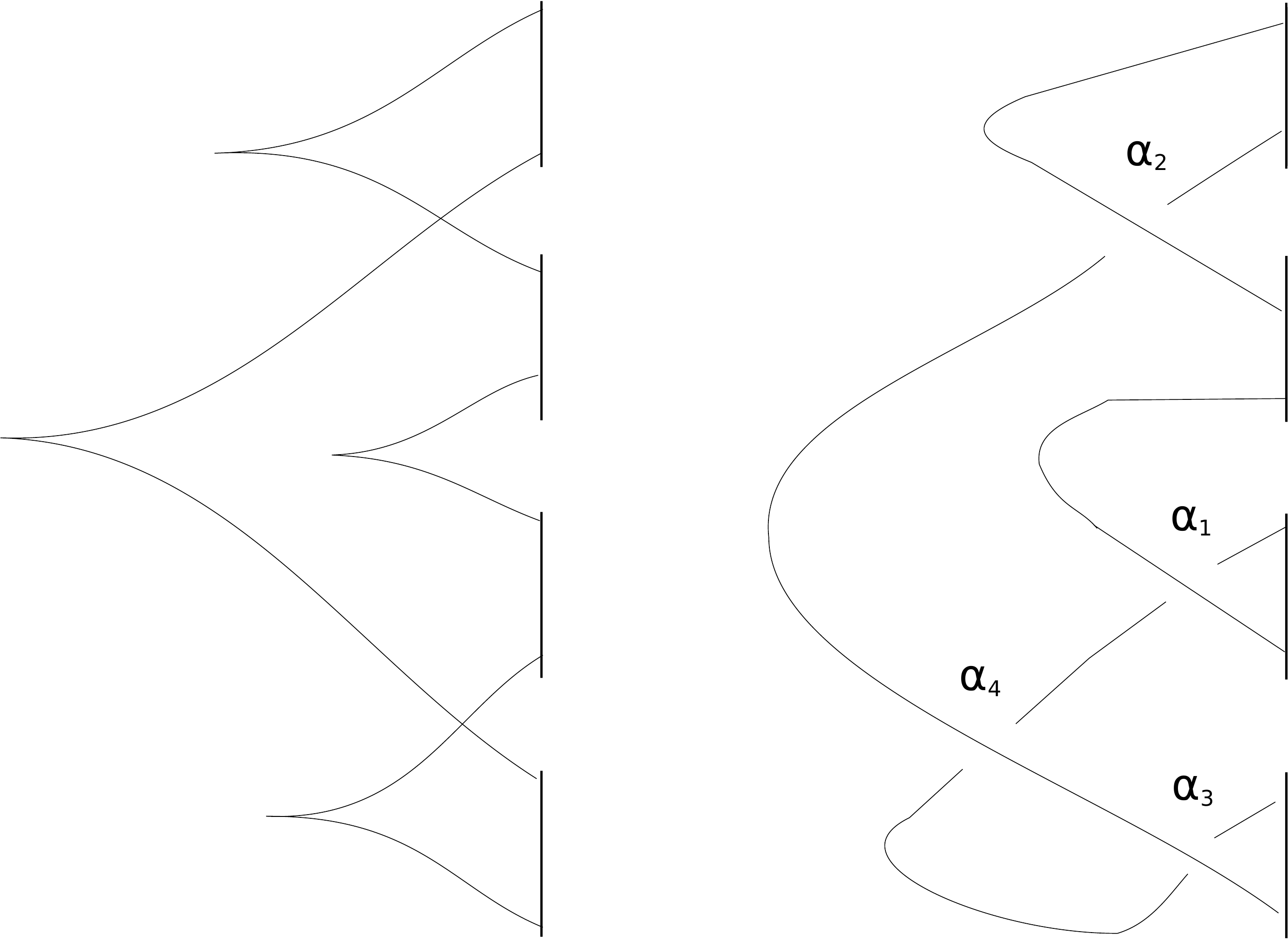}
\caption{The front diagram and resolution of a zero homologous knot that is not zero homotopic.}
\label{nollhomolog}
\end{figure}
\newpage
\section*{acknowledgement}
I would like to thank Tobias Ekholm for introducing me to the area of contact geometry and for many interesting conversations. I would also like to thank an anonymous referee for many helpful comments.


\begin{thebibliography}{MMMM}

\addtocounter{bibno}{1}
\bibitem{beeeffect}
F.Bourgeois, T.Ekholm and Y.Eliashberg, Effect of Legendrian Surgery,
[arXiv:math.SG/0911.0026v4]
\addtocounter{bibno}{1}
\bibitem{beesymp}
F.Bourgeois, T.Ekholm and Y.Eliashberg, Symplectic homology product via Legendrian surgery, Proc Natl Acad Sci USA 108, 8114-8121, (2011)
\addtocounter{bibno}{1}
\bibitem{CHEK}
Y.Chekanov, New invariants of Legendrian knots. 
(English summary) European Congress of Mathematics, Vol. II (Barcelona, 2000), 525–534, Progr. Math., 202, Birkhäuser, Basel, 2001. 
\addtocounter{bibno}{1}

\bibitem{EGA}
Ya. Eliashberg, A. Givental, and H. Hofer, Introduction to Symplectic Field Theory, GAFA2000
(Tel Aviv, 1999), Geom. Funct. Anal., 2000, Special Volume, Part II, 560–673.
\addtocounter{bibno}{1}
\bibitem{ekholm1}
T.Ekholm, J.Etnyre and M.Sullivan, Legendrian contact homology in PxR, Transactions of the American Mathematical Society 359, no 7, 3301-3335 (2007)
\addtocounter{bibno}{1}
\bibitem{ekholmoich}
T.Ekholm, J.Etnyre and M.Sullivan, Orientations in contact homology and double points of exact Lagrangian immersions; Internat. J. Math. 16, no 5, 453-532 (2005)
\addtocounter{bibno}{1}
\bibitem{ekholmr2n}
T.Ekholm, J.Etnyre and M.Sullivan, The contact homology of Legendrian submanifolds in R2n+1; Journal of Differential Geometry 71, no 2, 177-305 (2005)
\addtocounter{bibno}{1}
\bibitem{ELIA}
Y. Eliashberg. Invariants in contact topology. In Proceedings of the International Congress of Mathematicians, Vol. II (Berlin, 1998), number Extra Vol. II, pages 327–338 (electronic), 1998.
\addtocounter{bibno}{1}
\bibitem{SABLOFF2}
J. Etnyre, L. Ng and J. Sabloff. Invariants of Legendrian Knots and Coherent Orientations. J. Sympl. Geom. 1(2002), 321-367.

\addtocounter{bibno}{1}

\bibitem{NG1}
Lenhard L. Ng. Computable Legendrian invariants. Topology, 42(1):55–82, 2003.
\addtocounter{bibno}{1}
\bibitem{NGTREY}

Ng, Lenhard; Traynor, Lisa Legendrian solid-torus links. J. Symplectic Geom. 2 (2004), no. 3, 411–443.
\addtocounter{bibno}{1}
\bibitem{SABLOFF1}
J. Sabloff. Invariants of Legendrian knots in circle bundles. Commun. Contemp. Math. 5(2003), 569-627.
\addtocounter{bibno}{1}

\bibitem{WEIN}
A. Weinstein, Contact surgery and symplectic handlebodies, Hokkiado Math.
Journal 20 (1991) 241-251

%\addtocounter{bibno}{1}
%\bibitem{DRO}
%YU.V. Drobotukhina. An analogue of the Jones polynomial for links in $\R P^3$ and a generalization of the Kauffman-Murasugi theorem,
%Leningrad Math. J. Vol. 2 (1991), No. 3.
%\addtocounter{bibno}{2}
%\bibitem{DRO2}
%YU.V. Drobotukhina. Classification of Links in $\R P^3$ with at Most Six Crossings,
%Advances in Soviet mathematics. Volume 18 1994
%\addtocounter{bibno}{1}
%\bibitem{ROLF}
%D.Rolfsen. Knots and links,
%Houston : Publish or Perish, 1990
%\addtocounter{bibno}{1}
%\bibitem{KUPER}
%G.Kuperberg. Quadrisecants of knots and links, Theorem 1.3
%UC Berkeley 1988 [arXiv:math.GT/9712205]
%\addtocounter{bibno}{1}
%\bibitem{ROHK}
%V. A. Rokhlin, Complex topological characteristics of real algebraic curves, Uspekhi
%Mat. Nauk 33 (1978), no. 5, 77?89; English transl. Russian Math. Surveys 33 (1978),
%no. 5, 85?98.
%\addtocounter{bibno}{1}
%\bibitem{ORMI}
%Talk given during the conference ''Perspectives in analysis, geometry and topology`` in Stockholm 2008 by Orevkov titled ''Classification of algebraic links in $\R P^3$ of degree 5 and 6''.
%\addtocounter{bibno}{1}
%\bibitem{VASS}
%V.Vassiliev. On the spaces of polynomial knots,
%[arXiv:q-alg/9505008]
%\addtocounter{bibno}{1}
%\bibitem{VIRO}
%O.Viro. Encomplexing the writhe,
%[arXiv:math.GT/0005162v1]


\end{thebibliography}
\end{document}